\documentclass[11pt,reqno]{amsart}
\numberwithin{equation}{section}
\usepackage{amsmath,amsthm}
\usepackage{mathptmx}  
\usepackage{amssymb}
\usepackage{multicol}
\DeclareFontFamily{U}{BOONDOX-calo}{\skewchar\font=45 }
\DeclareFontShape{U}{BOONDOX-calo}{m}{n}{
  <-> s*[1.05] BOONDOX-r-calo}{}
\DeclareFontShape{U}{BOONDOX-calo}{b}{n}{
  <-> s*[1.05] BOONDOX-b-calo}{}
\DeclareMathAlphabet{\mathcalboondox}{U}{BOONDOX-calo}{m}{n}
\SetMathAlphabet{\mathcalboondox}{bold}{U}{BOONDOX-calo}{b}{n}
\DeclareMathAlphabet{\mathbcalboondox}{U}{BOONDOX-calo}{b}{n}
\usepackage{mathtools}
\usepackage{xfrac}

\newcommand{\mcb}[1]{{\mathcalboondox #1}}
\usepackage{amssymb}
\usepackage{amsmath}
\usepackage{amsthm}
\usepackage{diagbox}
\usepackage{booktabs}
\usepackage{enumitem}
\usepackage[sans]{dsfont} % allows \mathds{E 1}, without sans is less heavy 
\usepackage{bbm}
\usepackage{changebar}
\usepackage[colorlinks]{hyperref}
\usepackage{a4wide}
\usepackage[usenames,dvipsnames,svgnames,table]{xcolor}
\usepackage{xcolor}

\usepackage{tikz}
\usetikzlibrary{arrows,shapes,automata,petri,positioning,calc}
\tikzset{
    place/.style={
        circle,
        thick,
        draw=black,
        fill=gray!50,
        minimum size=20mm,
    },
        state/.style={
        circle,
        thick,
        draw=blue!75,
        fill=blue!20,
        minimum size=20mm,
    },
}

\tikzset{
    cross/.pic = {
    \draw[rotate = 45] (-0.2,0) -- (0.2,0);
    \draw[rotate = 45] (0,-0.2) -- (0, 0.2);
    }
}

\usepackage{ulem}
\usepackage{setspace}
\newtheorem{thm}{Theorem}[section]
\newtheorem{lem}[thm]{Lemma}
\newtheorem{cor}[thm]{Corollary}

\newtheorem{prop}[thm]{Proposition}

\newtheorem{definition}[thm]{Definition}

\newtheorem{rem}[thm]{Remark}

\usepackage{comment}

\title[{ Linear and Nonlinear Fractional PDEs from interacting particle systems}
]{Linear and Nonlinear Fractional PDEs from interacting particle systems }
\author{Pedro Cardoso, Patr\'icia   Gon\c calves}
\newcommand{\Addresses}{{% additional braces for segregating \footnotesize
		\footnotesize
		Pedro Cardoso, \textsc{\noindent Institute for Applied Mathematics \\
			University of Bonn \\
			Endenicher Allee, no. 60, 53115 Bonn, Germany}\par\nopagebreak
		\textit{E-mail address}: \texttt{pgondimc@uni-bonn.de}
		
		\medskip
		
		Patr\'icia   Gon\c calves, \textsc{\noindent Center for Mathematical Analysis,  Geometry and Dynamical Systems \\
Instituto Superior T\'ecnico, Universidade de Lisboa\\
Av. Rovisco Pais, no. 1, 1049-001 Lisboa, Portugal}\par\nopagebreak
		\textit{E-mail address}: \texttt{pgoncalves@tecnico.ulisboa.pt}
}}
\begin{document}
\begin{abstract}
In these notes, we describe the strategy for the derivation of the hydrodynamic limit for a family of long range interacting particle systems of exclusion type with symmetric rates. For $m \in \mathbb{N}:=\{1, 2, \ldots\}$ fixed, the hydrodynamic equation is $\partial_t \rho(t,u)= [-(-\Delta)^{\gamma /2} \rho^m](t,u) $. For $m=1$, this {is} the fractional equation, which is linear. On the other hand, for $m \geq 2$, this  is the fractional porous medium equation (which is nonlinear), obtained by choosing a rate which depends on the number of particles next to the initial and final position of a jump.
\end{abstract}
\maketitle

\section{Introduction}
\label{sec:1}
One important problem in statistical mechanics consists in deriving the macroscopic  evolution equations of the thermodynamic quantity (ies) of a gas from the interaction of its constituent models. According to Boltzmann, one should first find the equilibrium states of the system and then analyse the evolution out of the equilibrium. If we assume that the molecules perform continuous time random walks with certain constrains, then the evolution equations can be deduced for some dynamics. 
To answer this question, in the seventies, Spitzer in \cite{spitzer} introduced stochastic interacting particle systems (SIPS) as toy models for a variety of phenomena. 

Over the last fours decades, there has been remarkable progress in deriving those macroscopic equations  by means of rigorous mathematical results. The way to achieve it is to make a connection through a scaling parameter that links the continuous space where the solutions will be defined and the discrete space where the particles will evolve.   In this framework, many partial differential equations (PDEs) have been studied and derived from several underlying SIPS. The list of PDEs is quite vast and the same PDE can be obtained from different macroscopic dynamics. Nevertheless, the nature of these equations  highly depends on the chosen microscopic stochastic dynamics: it can be parabolic, hyperbolic, or even of fractional form. 

In these notes we focus on the derivation of fractional nonlinear PDEs, though we also discuss the fractional heat equation. More precisely, we are particularly interested in the derivation of the 
fractional porous medium equation (FPME) given for $\gamma\in(0,2)$ and $2 \leq m\in\mathbb N$ by
\begin{equation}\label{fractionalPME}
	\begin{cases}
		&\partial_{t}\rho(t,u) =  [-(-\Delta)^{\gamma/2} \rho^m](t,u) , \;\; u\in \mathbb{R}, \; t \in [0,T], \\
		&\rho(0,u) = g(u), \;\; u\in \mathbb{R},
	\end{cases}
\end{equation}
\noindent where $g:\mathbb R \rightarrow [0,1]$ is a measurable function. Above, $-(-\Delta)^{\gamma/2}$ denotes the fractional Laplacian operator, which, contrarily to the usual Laplacian operator, is nonlocal; see \eqref{eq:frac_lap} for its definition. The diffusive porous medium equation (DPME) is given by$\partial_{t}\rho(t,u) =  \Delta \rho^m(t,u)$ where $m \geq 2$ and models the diffusion of a substance through a porous medium. The quantity $\rho(t,u)$ denotes the density of the substance at time $t$ and position $u$.  Writing the equation in divergence form $\partial_{t}\rho(t,u) =  \nabla (D(\rho) \nabla  \rho (t,u)) $, the diffusion coefficient is $D(\rho)=m\rho^{m-1}$.

Since the diffusion coefficient $D(\rho)$ vanishes as the density tends to zero, this suggests that at the microscopic level, when the density of particles is low, it becomes more difficult for particles to diffuse. With this scenario in mind, in \cite{GLT} it was proposed a model of interacting particles that has a hydrodynamic limit given by the DPME, without boundary conditions. The power $m$ derived in that article is an integer number since it is related to the required number of particles close to the sites where an exchange occurs: when $m$ increases, the same occurs with the necessary quantity of particles. The DPME was also obtained in \cite{bonorino}; the novelty in that work was the presence of Dirichlet, Robin and Neumann boundary conditions.

So far, various properties of the DPME are classical; we mention a couple of them here. An interesting feature is that if the initial condition is non negative then the same is true for the solution at any time $t \geq 0$. Another characteristic is that solutions of DPME evolve with finite speed of propagation (in opposition to the heat equation, whose solutions spread out with infinite speed) : see, for example, the Barenblatt solution as an example of the finite speed of propagation property. The FPME is less studied in the literature than its diffusive counterpart. Despite that, we refer the interested reader to \cite{DQRV1,DQRV2}, where the authors studied properties of the weak  solutions of the FPME.

Very recently in \cite{CDG} the FPME (for $m=2$) was obtained as the hydrodynamic limit of a particle system. Our goal in these notes is to generalize the results of \cite{CDG} to any $m \in \mathbb{N}$, presenting the derivation of \eqref{fractionalPME} as the  hydrodynamic limit of a collection of SIPS whose jump rates are dependent on the number of particles close to where the jump occurs. In every model of this collection the exclusion rule is enforced, i.e., jumps to occupied sites are suppressed. With this dynamics, each site is occupied by at most one particle, therefore we are always dealing with exclusion process, highly studied in the SIPS literature. The case $m=1$ (where the fractional heat equation is derived) was first presented in \cite{milton} and generalized in \cite{CGJ2} in the presence of a slow barrier (where fractional Neumann and Robin boundary conditions were obtained). Since the linear setting was introduced a while ago in \cite{milton}, we focus on the nonlinear one in this work.

Besides the exclusion rule, for $m \geq 2$ jumps from $x$ to $y$ are only allowed if there exists a set (among some possibilities, depending on $x$ and $y$) of $m-1$ sites at the neighborhood of the initial and final position which is fully occupied. We then show that the FPME can be obtained as the hydrodynamic limit of this SIPS. We do so by employing the tools of the classical entropy method developed in \cite{GPV}. A direct consequence of this approach is the existence of weak solutions for the FPME, but we still need to prove the uniqueness of those solutions. We highlight that the same general strategy is applied to the case $m=1$ (as it was done in \cite{milton} and \cite{CGJ2}).  In these notes, by a matter of size we could not give all the technical details of the proofs of our results. Instead, we tried to choose the ones which are the most relevant and most enlightening for the case $m \geq 2$.

Here follows some open problems that arise naturally from our study; the first one is regarding the derivation of \eqref{fractionalPME} when $m\notin\mathbb N$. Recently in \cite{GNS} the answer to the previous question was provided for the diffusive case and for $m\in(0,2)$; and the constructed model is a kind of reinforcement/penalization of the symmetric simple exclusion process. We believe that similar arguments could provide good answers in the fractional case and we leave this to future work. 
Another compelling question has to do with the extension of our results to higher dimensions, and also adding different boundary conditions such as Dirichlet, Robin or Neumann to the hydrodynamic equation. So far, fractional boundary conditions have been obtained in \cite{stefano} and \cite{CGJ2}, but these works only deal with the linear case. We are also interested in being able to apply the results we have to extend the list of fractional PDEs that one can obtain from the hydrodynamic limit of SIPS: either by choosing an exclusion process that allows more than a particle per site or by considering systems of SIPS in order to obtain systems of coupled FPMEs.
\medskip 

\textbf{Outline of the article:} In Section \ref{sec:2} we introduce the microscopic models that we analyse, we provide the notion of weak solutions for the fractional equations that we obtain and we state our main result, namely Theorem \ref{hydlim}. In Section \ref{sec:heuristics} we give an heuristic argument to illustrate how the notion of weak solution of our equations arises from the underlying particle system. In order to do so, several intermediate results are assumed :we refer the reader to \cite{BGJ2, CDG, CGJ2} for the proofs. In Section \ref{sec:main_theo} we prove our main theorem, i.e. the hydrodynamic limit. More precisely, in Subsections \ref{sec:tight}, \ref{sec:charac} and \ref{sec:EE} we obtain tightness, characterize the limit points and prove some energy estimates. Finally, in Subsection \ref{sec:replem} we derive some results that are needed along the arguments in the previous subsections.  

\section{Models and Statement of Results }
\label{sec:2}

\subsection{Microscopic models: exclusion and porous medium}
\label{subsec:2}
In this section we introduce the two models that we will analyse in these notes. The first one is the well known symmetric exclusion process with long range interactions. The second is again an exclusion process, but the exchange rate depends strongly on the occupancy close to the sites where the exchanges occur. Both processes will be evolving in $\mathbb{Z}$ the set of integer numbers, whose elements are called \textit{sites} and denoted by Latin letters such as $x,y,z$. 
Now  we explain in detail the two dynamics. 
First we observe that our processes belong to the collection  of exclusion processes which only allow at most one particle per site. Therefore, the state  space of our models is the set $\Omega = \{0,1\}^{\mathbb{Z}}$. The elements of this space are called configurations and we denote them by  Greek letters such as $\eta$.  Given a configuration $\eta \in \Omega$ and a site $x \in \mathbb{Z}$,  we can only have two values for $\eta(x)$, i.e. $\eta(x)\in\{0,1\}$ and we interpret  that the site $x$ is empty if $\eta(x)=0$, and that the site $x$ is occupied if $\eta(x)=1$. 

Particles jump in $\mathbb Z$ according to a probability measure $p:\mathbb{Z} \rightarrow [0,1]$ defined  by 
\begin{equation}\label{transition prob}
	\forall z \in \mathbb{Z}, \quad p(z) = \frac{c_{\gamma}}{ |z|^{-\gamma-1}}   \mathbbm{1}_{ \{z \neq 0 \} }, 
\end{equation}
where $\gamma>0$ is fixed and $c_{\gamma}$ is a normalizing constant that turns $p(\cdot)$ into a probability.

Given a configuration $\eta$ and sites $x,y\in\mathbb Z$, we denote by $\eta^{x,y}$, the configuration obtained from $\eta$ where the occupation of the sites $x$ and $y$ are exchanged, i.e.  
\begin{equation*}
	\eta^{x,y}(z) := 
	\begin{cases}
		\eta(z), \quad z \ne x,y,\\
		\eta(y), \quad z=x,\\
		\eta(x), \quad z=y.
	\end{cases}
\end{equation*}

To properly define the infinitesimal generators of our Markov processes we need to introduce the notion of local functions. Therefore,  $f: \Omega \rightarrow  \mathbb{R}$ is a \textit{local function}, if there exists a finite $\Lambda \subset \mathbb{Z}$ such that  { 
	$
	\forall x \in \Lambda, \eta_1(x) = \eta_2(x)  \Rightarrow f(\eta_1) = f(\eta_2).
	$} 
With this definition in hands, for $m \in \mathbb{N}$ fixed, the continuous time Markov processes that we consider and that we denote by  $(\eta_{t})_{t\geq 0}$ have infinitesimal generator $\mcb L$  given on local functions $f: \Omega \rightarrow  \mathbb{R}$ by 
\begin{align}
	(\mcb L f)(\eta) :=  &\frac{1}{4}\sum_{x,y } p(y-x){c}^{(m)}_{x,y}(\eta)   \xi_{x,y}(\eta)[f(\eta^{x,y})-f(\eta)] \nonumber \\
	=& \frac{1}{4}\sum_{x,y }  p(y-x){c}^{(m)}_{x,y}(\eta) [f(\eta^{x,y})-f(\eta)]. \label{geninf}
\end{align}
Above, $\xi_{x,y}(\eta):=\eta(x) [1 - \eta(y)] + \eta(y) [ 1 - \eta(x) ] = [\eta(y) - \eta(y)]^2$, for any $x,y \in \mathbb{Z}$. The equality in \eqref{geninf} holds due to $\xi_{x,y}(\eta)[f(\eta^{x,y})-f(\eta)]=f(\eta^{x,y})-f(\eta)$. Above and hereinafter we will always assume that the discrete variables in a summation range over $\mathbb{Z}$, unless it is stated otherwise.

Before stating the general expression of ${c}^{(m)}_{x,y}(\eta)$, we give some motivation for its definition. For $m=1$, in order to recover the dynamics described in \cite{milton} and in \cite{CGJ2} (in the later for the choice  $\alpha=1$ and $\beta=0$), we define  ${c}^{(1)}_{x,y}(\eta)=2$. In this work, ${c}^{(m)}_{x,y}(\eta)$ is always equals to twice the rate of a exchange of particles between $x$ and $y$, according to the configuration $\eta$. For $m \geq 2$, in order to produce the diffusive porous medium equation $\partial \rho = \Delta \rho^m$, the role of ${c}^{(m)}_{x,y}(\eta)$ in \eqref{geninf} is fulfilled by
\begin{equation} \label{ratespordif}
	{c}^{(m, dif)}_{x,x+1}(\eta) := \sum_{k=1}^m \underset{j \notin \{0, 1 \} }{\prod_{j=k-m}^k} \eta(x + j).
\end{equation}
For details see  \cite{bonorino}. In particular, we get
\begin{align*}
	{c}^{(2, dif)}_{x,x+1}(\eta)& = \eta(x-1) + \eta(x+2) = \eta(x-1) + \eta \big( (x+1) + 1 \big);
	\\
	%{c}^{(3, dif)}_{x,x+1}(\eta) &= \eta(x-2) \eta(x-1) + \eta(x-1) \eta(x+2) + \eta(x+2) \eta(x+3)  \\
	%=& \eta(x-2) \eta(x-1) + \eta(x-1) \eta \big( (x+1) + 1\big) + \eta \big( (x+1) + 1\big) \eta \big( (x+1) + 2 \big); \\
	{c}^{(3, dif)}_{x,x+1}(\eta) =& \eta(x-2) \eta(x-1) +  \eta \big( (x+1) + 1\big) \big[ \eta(x-1) +  \eta \big( (x+1) + 2 \big)\big];\\
	%{c}^{(4, dif)}_{x,x+1}(\eta) &= \eta(x-3) \eta(x-2) \eta(x-1) +  \eta(x-2) \eta(x-1) \eta(x+2) \\
	%+& \eta(x-1) \eta(x+2) \eta(x+3) + \eta(x+2) \eta(x+3) \eta(x+4)  \\
	%=& \eta(x-3) \eta(x-2) \eta(x-1) +  \eta(x-2) \eta(x-1) \eta \big( (x+1) + 1\big) \\
	%+& \eta(x-1) \eta \big( (x+1) + 1\big) \eta \big( (x+1) + 2\big) \\
	%+& \eta \big( (x+1) + 1\big) \eta \big( (x+1) + 2\big) \eta \big( (x+1) + 3\big).
	{c}^{(4, dif)}_{x,x+1}(\eta) =& \eta(x-3) \eta(x-2) \eta(x-1) +  \eta(x-2) \eta(x-1) \eta \big( (x+1) + 1\big) \\
	+&  \eta \big( (x+1) + 1\big) \eta \big( (x+1) + 2\big) \big[ \eta(x-1) + \eta \big( (x+1) + 3\big) \big].
\end{align*}
In particular, replacing $x+1$ by $y$, we get 
\begin{align*}
	{c}^{(2, dif)}_{x,y}(\eta) &= \eta(x-1) + \eta ( y + 1 );\\
	{c}^{(3, dif)}_{x,y}(\eta) &= \eta(x-2) \eta(x-1) + \eta(x-1) \eta (y + 1 ) + \eta ( y + 1 ) \eta ( y + 2);
	\\
	{c}^{(4, dif)}_{x,y}(\eta) =& \eta(x-3) \eta(x-2) \eta(x-1) +  \eta(x-2) \eta(x-1) \eta ( y+ 1 ) \\
	+& \eta(x-1) \eta ( y + 1 ) \eta ( y + 2 ) + \eta ( y + 1) \eta ( y + 2 ) \eta ( y + 3 ).
\end{align*}
In order to have symmetric rates (i.e., ${c}^{(m)}_{x,y}(\eta)={c}^{(m)}_{y,x}(\eta)$), in \cite{CDG} we  made the choice  
\begin{align*}
	{c}^{(2)}_{x,y}(\eta) := & {c}^{(2, dif)}_{x,y}(\eta) + {c}^{(2, dif)}_{y,x}(\eta) 
	= \eta(x-1) + \eta(y+1) + \eta(y-1) + \eta(x+1).  
\end{align*}
Note that in last  case ($m=2$) the jump rate can be multiplied by a factor $r\in\{0,1,2,3,4\}$ depending on the occupancy of the sites $\{x-1,x+1,y-1,y+1\}$, see the figure below.  

\begin{figure}[htb]
	\begin{center}
		\begin{tikzpicture}[thick, scale=0.75]
			\draw[latex-] (-8.5,0) -- (6.3,0) ;
			\draw[-latex] (-8.5,0) -- (6.3,0) ;
			\foreach \x in  {-8,-7,-6,-5,-4,-3,-2,-1,0,1,2,3,4,5,6}
			\draw[shift={(\x,0)},color=black] (0pt,0pt) -- (0pt,-3pt) node[below] 
			{};

			\draw[] (-8,-0.1) node[below] {\footnotesize{$x-1$}};
			\draw[] (-7,-0.1) node[below] {\footnotesize{$x$}};
			\draw[] (-6,-0.1)  node[below] {\footnotesize{$x+1$}};
			\draw[] (-3,-0.1)  node[below] {\footnotesize{$z-1$}};
			\draw[] (-2,-0.1)  node[below] {\footnotesize{$z$}};
			\draw[] (-1,-0.1)  node[below] {\footnotesize{$z+1$}};
			\draw[] (1,-0.1) node[below] {\footnotesize{$y-2$}};
			\draw[] (2,-0.1) node[below] {\footnotesize{$y-1$}};
			\draw[] (3,-0.1) node[below] {\footnotesize{$y$}};
			\draw[] (4,-0.1)  node[below] {\footnotesize{$y+1$}};
			\draw[] (5,-0.1)  node[below] {\footnotesize{$y+2$}};
			\draw[] (6,-0.1)  node[below] {\footnotesize{$y+3$}};
			\draw [thick, teal]  (-1,1.15) -- (-1.4,0.75) ; 
			\draw [thick, teal]  (-1,0.75) -- (-1.4,1.15) ; 
			\draw [thick, teal]  (1,1.25) -- (1.4,0.7) ; 
			\draw [thick, teal]  (1,0.7) -- (1.4,1.25) ;

			\node[shape=circle,minimum size=0.5cm] (F) at (6,0.3) {};
			\node[shape=circle,minimum size=0.5cm] (G) at (5,0.3) {};
			\node[shape=circle,minimum size=0.5cm] (H) at (3,0.3) {};
			\node[shape=circle,minimum size=0.5cm] (I) at (1,0.3) {};
			\node[shape=circle,minimum size=0.5cm] (J) at (-2,0.3) {};

			\node[ball color=black!30!, shape=circle, minimum size=0.4cm] (A) at (-8,0.3) {};
			\node[ball color=black!30!, shape=circle, minimum size=0.4cm] (B) at (-7,0.3) {};
			\node[ball color=black!30!, shape=circle, minimum size=0.4cm] (C) at (-6.,0.3) {};
			\node[ball color=black!30!, shape=circle, minimum size=0.4cm] (D) at (2,0.3) {};
			\node[ball color=black!30!, shape=circle, minimum size=0.4cm] (E) at (4.,0.3) {};

			%	\path [->] (B) edge[bend left =75] node[above] {\textcolor{black}{$ p(y+3-x)   $}} (F);           
			\path [->] [red] (B) edge[bend left=62] node[above] {\textcolor{red}{$ 3p(y+2-x) / 2 $}} (G);        
			\path [->] [blue] (B) edge[bend left=55] node[below]  { \textcolor{blue}{$2p(y-x)$}} (H);
			\path [->] (C) edge[bend left=45] node[above]  {\textcolor{black}{$p\big(  z - (x+1) \big) / 2$}} (J);
			%	\path [->] (D) edge[bend right=25] node[above]  {\textcolor{black}{$p(-1) / 2  $}} (I);
			\path [->] [teal] (D) edge[bend right=45] node[above]  {\textcolor{teal}{$ p\big(  z - (y-1) \big) \cdot 0   $}} (J);
			
		\end{tikzpicture}
		\bigskip
		\caption{Fractional porous medium model with long jumps. The rate of a jump between $x_0$ and $y_0$ is equal to half of the number of sites in $\{x_0-1, x_0+1, y_0-1, y_0+1\}$ which are occupied. For instance, the jump from $x$ to $y$ has rate $4/2$, since all the neighbors of $x$ and $y$ are occupied. On the other hand, the jump from $y-1$ to $z$ has rate zero (even $z$ being an empty site), since neither of the sites $z$ or $y-1$ is next to an occupied site. }\label{figure7int}
	\end{center}
\end{figure}
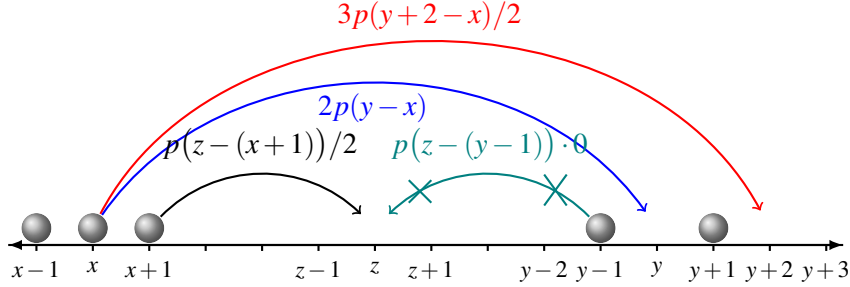

%	\pg{PEdro coloca a figura pf.}

In \cite{CDG}, the following remark was crucial.
\begin{rem} \label{remsep}
	For the particular case where $y = x+1$, we get
	\begin{align*}
		\xi_{x,y}(\eta) {c}^{(2)}_{x,y}(\eta) =& \xi_{x,x+1}(\eta) [ \eta(x-1) + \eta(x) + \eta(x+1) + \eta(x+2)    ] \\
		= & \xi_{x,x+1}(\eta) [ \eta(x-1) + \eta(x+2)   +1 ] \geq  \xi_{x,x+1}(\eta).
	\end{align*}
	Above we used the fact that $\eta(z) \in \{0,1\}$ for every $z \in \mathbb{Z}$. In particular, the porous medium dynamics (for $m=2$) always allow jumps of size $1$, as long as the exclusion rule is enforced.
\end{rem}
However, for every $m \geq 3$  it is \textit{not true} that 
\begin{align*}
	\forall x \in \mathbb{Z}, \forall \eta \in \Omega, \quad \xi_{x,x+1}(\eta)[ {c}^{(m, dif)}_{x,x+1}(\eta) + {c}^{(m, dif)}_{x+1,x}(\eta) \big] \geq \xi_{x,x+1}(\eta).
\end{align*}
Therefore, in order to assure that jumps of size $1$ are always allowed (which is fundamental to avoid blocked configurations), if $m \geq 3$ we add (a multiple of) the rates of the  symmetric simple exclusion process in our dynamics. Then finally, we can give the general expression for ${c}^{(m)}_{x,y}(\eta)$.
\begin{equation} \label{defrates}
	{c}^{(m)}_{x,y}(\eta): = 
	\begin{cases}
		2, \quad & m = 1; \\
		{c}^{(m, dif)}_{x,y}(\eta) + {c}^{(m, dif)}_{y,x}(\eta) + \mathbbm{1}_{ \{ m \geq 3 \} } \mathbbm{1}_{ \{ |x - y| = 1 \} }, \quad & m \geq 2.  
	\end{cases}
\end{equation}
\begin{rem} \label{ratpor}
	For every $m \geq 2$, it holds
	\begin{align*} % \label{eq:rates_porous}
		{c}^{(m, dif)}_{x,y}(\eta) + {c}^{(m, dif)}_{y,x}(\eta):=\sum_{j=1}^{m} \Big( \prod_{i=0}^{m-2} \eta(a_{i+j,x,y})  +\prod_{i=0}^{m-2} \eta(a_{i+j+2(m-1),x,y}) \Big),
	\end{align*}
	where for every $x,y \in \mathbb{Z}$, $\{a_{1,x,y}, a_{2,x,y}, \ldots a_{4m-4,x,y}\} \subset \mathbb{Z}$ is defined by
	\begin{equation*}% \label{defajxy}
		a_{j,x,y}:=
		\begin{cases}
			x-(m-j), &j \in \{1, \ldots, m-1\}; \\
			y+ j - (m-1),  &j \in \{(m-1)+1, \ldots, 2(m-1)\}; \\
			y- (3m-2 - j), &j \in \{2(m-1)+1, \ldots, 3(m-1)\}; \\
			x+ j - 3(m-1),  &j \in \{3(m-1)+1, \ldots, 4(m-1)\}.
		\end{cases}
	\end{equation*}
	This means that if $|x-y| > 1$, a particle can only jump between $x$ and $y$ if at least one of the $2m$ windows $\{a_{1,x,y}, \ldots ,a_{m-1,x,y}\}, \ldots, \{a_{m,x,y}, \ldots ,a_{2m-2,x,y}\}$, $\{a_{2m-1,x,y}, \ldots ,a_{3m-3,x,y}\}, \ldots, \{a_{3m-2,x,y}, \ldots ,a_{4m-4,x,y}\}$ is fully occupied.  
\end{rem}

An important feature of all  these models is that their invariant measures are the  Bernoulli product measures with  parameter $b \in (0,1)$ denoted by $\nu_b$ and defined on $\Omega$ by their  marginals given on $y \in \mathbb{Z}$ by $\nu_b\{ \eta \in \Omega: \eta(y)= 1\}=b.$
This means that under $\nu_b$, the random variables $(\eta(y))_{y \in \mathbb{Z}}$ are independent and all have  Bernoulli distribution of parameter $b$.  Moreover, for every $\eta \in \Omega$ and for every $x,y \in \mathbb{Z}$ it holds 
\begin{equation} \label{invexch}
	\forall x,y \in \mathbb{Z}, \forall \eta \in \Omega, \quad \nu_b (\eta^{x,y}) = \nu_b (\eta).
\end{equation}
Thus, since  $p(\cdot)$ given in \eqref{transition prob} is symmetric, the measure $\nu_b$ is reversible and in particular, it is also invariant. 
\subsection{Hydrodynamic Equations}
\label{subhydeq}

For $\gamma\in(0,2)$, the fractional Laplacian $-(-\Delta)^{\gamma/2}$ of exponent $\gamma/2$ is defined on the set of functions $G:\mathbb{R} \rightarrow \mathbb{R}$ such that
$
\int_{\mathbb{R}} \frac{G(u)}{(1+|u|)^{1+\gamma}}\, du < \infty
$
by
\begin{equation}\label{eq:frac_lap}
	[-(-\Delta)^{\gamma/2}G](u) := c_{\gamma} \lim_{\varepsilon \rightarrow 0^{+}} \int_{\mathbb{R}} \mathbbm{1}_{\{ |u-v|\geq \varepsilon \} }\frac{G(v)-G(u)}{|u-v|^{1+\gamma}} \, dv
\end{equation}
provided the limit exists.  Above, $c_\gamma$ is the constant appearing in \eqref{transition prob}. 
An equivalent definition for the fractional Laplacian given in last display  is through the Fourier transform, i.e. $\widehat{-(-\Delta)^{\gamma/2}}G(\xi)=|\xi|^\gamma\widehat G(\xi)$, nevertheless, we will not use this definition in what follows. 
Fix a measurable function $g:\mathbb{R} \rightarrow [0,1]$. In this article we are interested in analysing solutions {of} the equation \eqref{fractionalPME}. In that equation we assume that $m\in\mathbb N$. When $m=1$, the equation \eqref{fractionalPME} is the  fractional heat equation, which is a linear equation; while for $2\leq m\in\mathbb N$ it is the porous medium equation; which is a nonlinear equation.  
Now we define the space where our solutions  belong.
\subsubsection{Sobolev spaces and weak solutions }
\begin{definition}
	The Sobolev space $\mcb {H}^{\gamma/2}$ in $\mathbb{R}$ consists of all functions $f \in L^2(\mathbb{R})$ such that 
	\begin{align*}
		[f]^2_{\mcb {H}^{\gamma/2}}:= \iint_{\mathbb{R}^2} \frac{[f(u)-f(v)]^2}{|u-v|^{1+\gamma}} \, du \, dv < \infty. 
	\end{align*}
	This is a Hilbert space for the norm $\| \cdot \|_{\mcb {H}^{\gamma/2}}$ defined by
	{
		\begin{align*}
			\| f \|^{2}_{\mcb {H}^{\gamma/2}} := \int_{\mathbb{R}} [f(u)]^2 du + [f]^2_{\mcb {H}^{\gamma/2}}.
		\end{align*}
	}
\end{definition}
Below we use the notation $\langle f,g\rangle$ to denote the inner product between two functions $f,g\in  L^2(\mathbb R)$; moreover, $N \subset L^2(\mathbb{R})$ is a normed vector space with norm $\| \cdot \|_{N}$.
\begin{definition}
	We denote by $L^{2}\left(0,T; N \right)$ the set of all measurable functions $f:[0,T] \times \mathbb{R} \rightarrow \mathbb{R}$ such that $f(s, \cdot) \in N$ for almost every $s$ on $[0,T]$ and
	$\int_{0}^{T}\| f(s, \cdot) \|^{2}_{N}\; ds < \infty .$ Moreover, the set $P ([0,T], N )$ is the space of functions $G: [0,T] \times \mathbb{R} \rightarrow \mathbb{R}$ such that there exist $k \in \{0, 1, 2, \ldots, \}$ and $G_0, G_1, \ldots, G_k \in N$ so that for all $(t,u) \in [0,T] \times \mathbb{R}$, $ G(t,u) = \sum_{j=0}^{k} t^{j} G_j(u).
	$
	
\end{definition}

{Given $r \in \{1, 2, \ldots\}$,  $G: \mathbb{R} \rightarrow \mathbb{R}$ is in $C^{r}(\mathbb{R})$ if $G$ is $r$ times continuously differentiable and for $r=0$, $C^0(\mathbb{R})$ denotes the set of continuous functions in $\mathbb{R}$. Also, $G \in C_{c}^{r}(\mathbb{R})$ if $G \in C^{r}(\mathbb{R})$ and $G$ has compact support. Moreover, we use the notation $C_c^{\infty}(\mathbb{R}):=\cap_{r=0}^{\infty} C_{c}^r (\mathbb{R})$.}
Our space of test functions is $\mcb S:= P \big( [0,T], C_c^{\infty}(\mathbb{R}) \big)$. %For every $G \in \mcb S$, we denote
%{
	%\begin{equation} \label{defbG}
	%	b_{G}:=\inf \Big\{ \tilde{b} \geq 0: \forall s \in [0,T] \quad \sup_{|u| \geq \tilde{b} }|G(s,u)|=0  \Big\}. 
	%\end{equation}
	%}
In what follows, given $H: [0,T] \times \mathbb{R} \rightarrow \mathbb{R}$, we denote $H(s, \cdot)$ by $H_s$, where $s \in [0,T]$. 

\begin{definition}\label{eq:dif}\textbf{(Weak solutions)}
	Let $g: \mathbb{R} \rightarrow [0,1]$ be a measurable function. We say that a function $\rho :[0,T] \times \mathbb{R} \rightarrow [0,1]$ is a weak solution of the fractional porous medium equation \eqref{fractionalPME} with initial condition $g$  if the following conditions hold:
	\begin{enumerate}
		\item
		for every $t \in [0,T]$ and for every $G \in \mcb S$, it holds  $F(t, \rho,G,  g)=0$, where
		\begin{equation}
			\begin{split} \label{eq:int_eq}
				F(t, \rho,G,  g):=&  \langle \rho_t, G_t\rangle - \langle   g, G_0\rangle \\
				-& \int_0^t \langle \rho_s, \partial_s G_s\rangle \, ds- \int_0^t \langle \rho^m_s,  [-(- \Delta)^{\gamma/2} G_s]\rangle \, ds; 
			\end{split}
		\end{equation}
		\item 
		there exists some $b \in (0,1)$ satisfying both $ \rho - b \in L^2 \big(0, T ; L^{2} ( \mathbb{R} ) \big)$ and also $ \rho^m - b^m \in L^2 (0, T ; \mcb{H}^{\gamma/2} )$.
	\end{enumerate}
\end{definition}
\begin{rem}
	In the particular case $m=1$, the condition $ \rho - b \in L^2 \big(0, T ; L^{2} ( \mathbb{R} ) \big)$ is a direct consequence of $ \rho^m - b^m \in L^2 (0, T ; \mcb{H}^{\gamma/2} )$.
\end{rem}
\begin{rem}
	The uniqueness of weak solutions as given in the previous definition is proved in Appendix C of \cite{CDG}.
\end{rem}
\subsection{Hydrodynamic Limit}
\label{subhlim}

Let us now fix a time $T > 0$ and a finite time horizon $[0,T]$. 
Since we want to observe a non-trivial evolution, we consider our Markov processes speeded up in the time scale $n^{\gamma}$ for $\gamma\in(0,2).$ We will explain latter the reason for the choice of the time scale defined above. 
We use the notation $\eta_t^n:=\eta_{t n^{\gamma}}$ for the speed up process and note that it has  the generator  $\mcb L_n:=n^{\gamma} \mcb L$.

Since in our models the unique conserved quantity is the density of particles, we define the empirical measure as the measure that gives weight $1/n$ to each particle in the following way:
\begin{equation*}
	\pi^{n}(\eta, du) := \frac{1}{n}\sum_{x}\eta(x)\delta_{x/n}(du),
\end{equation*}
where $\delta_u$ is a Dirac mass on $u \in \mathbb{R}$. 
Now we define the process of the empirical measures as $\pi^{n}_{t}(\eta,du) := \pi^{n}(\eta^n_{t},du).$
For a measurable function $G:\mathbb{R} \rightarrow \mathbb{R}$, we denote the integral of $G$ with respect to the empirical measure $\pi_{t}^{n}$, by $\langle \pi^{n}_{t},G \rangle.$

Now we define the conditions on the initial measures. To that end, for every $n \geq 1$, let $\mu_n$ be a probability measure on our state space $\Omega$. Let  $  \mcb {M}^+$ be the space of non-negative Radon measures on $\mathbb{R}$ and equipped with the weak topology. Let    $\mathbb{P}_{\mu_n}$ be the probability measure on the Skorokhod space $\mcb {D}([0,T],\Omega)$ induced by the  Markov process $(\eta_{t})_{t \in [0,T]}$ and the initial measure $\mu_n$. Moreover, let $\mathbb{Q}_{n}$ be the probability measure on $\mcb{D}([0,T],  \mcb {M}^{+})$ induced by $(\pi_{t}^{n})_{t \in [0,T]}$ and $\mathbb{P}_{\mu_n}$.

\begin{definition}\label{associated profile}
	Let $g: \mathbb{R}\rightarrow[0,1]$ be a measurable function and $(\mu_n )_{n\ \geq 1}$ a sequence of probability measures in $\Omega$. We say that $(\mu_n )_{n\ \geq 1}$ is \textit{associated with $ g(\cdot)$} if for any $G \in C_c^0(\mathbb{R})$ and any $\delta > 0$, 
	\begin{equation*}
		\lim _{n \rightarrow \infty } \mu_n \Big( \eta \in \Omega : \Big|  \frac{1}{n} \sum_{x} G \big(\tfrac{x}{n} \big) \eta(x)   - \int_{\mathbb{R}} G(u)g(u) \, du \, \Big|    > \delta \Big)= 0.
	\end{equation*} 
\end{definition}

The previous assumption requires the validity of a law of large numbers for the empirical measure at time $t=0$., i.e.  the sequence of  random measures $\{\pi^n_0\}_{n\in\mathbb N}$ converges, as $n \rightarrow \infty$ to the deterministic measure  $\pi_0(du):= g(u)du$.
Under last  assumption, our  goal is to show that for any time $t$ the same result is true but the density will be  a weak solution to a PDE, called the \textit{hydrodynamic equation.}

In some of our results below we will need some extra conditions on the initial measure of the systems. To that end we state the following condition, that will be imposed for some models:
%\begin{equation}\label{defcb}
%\textrm{\textbf{Assumption (A)}:  There exists}  C_b>0 \textrm{ such that}\, 
%\forall n \geq 1, 		H ( \mu_n | \nu_b) \leq C_b n.
%\end{equation}
\begin{equation}\label{defcb}
	\textrm{\textbf{Assumption (A)}:} \quad \exists C_b>0: \;  \forall n \geq 1, \quad		H ( \mu_n | \nu_b) \leq C_b n.
\end{equation}
Now we are in  position to state the main result of these notes.

\begin{thm}\label{hydlim}
	Let $ g: \mathbb{R} \rightarrow [0,1]$ be a measurable function. Let $(\mu_n)_{n \geq 1}$ be a sequence of probability measures in $\Omega$ associated to the profile $ g(\cdot)$ and satisfying Assumption (A)
	for some $b \in (0,1)$. Then, for any $t \in [0,T]$, any $G \in C_c^0(\mathbb{R})$ and any $\delta>0$,
	\begin{equation*}\label{limHidreform}
		\lim_{n \rightarrow \infty}\mathbb{P}_{\mu_n}\Big(  \eta_{\cdot}^n \in \mcb {D}([0,T], \Omega): \Big| \frac{1}{n} \sum_{x} G\big(\tfrac{x}{n}\big) \eta_t^n(x) - \int_{\mathbb{R}} G(u)\rho(t,u)\,du  \,  \Big| > \delta \Big)=0,
	\end{equation*}
	where $\rho(t, \cdot)$ is the unique weak solution of \eqref{fractionalPME}.
\end{thm}
\noindent
\textbf{Outline of the proof:} 
For all the models we follow the entropy method introduced in \cite{GPV}. 
To make the presentation as simple as possible we hide many technical results and we just focus on the most important steps of the proof. We refer the reader to \cite{BGJ2, CDG, CGJ2} (and references there in) for more details. 
In Section \ref{sec:heuristics} we give an heuristic argument, highlighting the main steps of the proof so that we can derive the integral notions of the weak solutions of our hydrodynamic equations. 
In Section \ref{sec:tight}, we prove that the sequence $(\mathbb{Q}_{n})_{n \geq 1}$ is tight with respect to the Skorokhod topology of $\mcb {D}([0,T],  \mcb{M}^{+})$.
Then, as a consequence of 
Prohorov's Theorem (see Theorem 6.1 in \cite{billingsley}), we are able to conclude that this sequence is relatively compact. 
From this we know that  the sequence  $(\mathbb{Q}_{n})_{n \geq 1}$ has a convergent subsequence, i.e., there exists a subsequence  $(\mathbb{Q}_{n_j})_{j \geq 1}$ and a measure $\mathbb{Q}$ such that  the sequence $(\mathbb{Q}_{n_j})_{j \geq 1}$ weakly converges to $\mathbb{Q}$. In Subsection \ref{sec:charac} (resp. Subsection \ref{sec:EE}) we prove that any such limit point $\mathbb{Q}$ is concentrated on trajectories of measures satisfying the first (resp. the second) condition of weak solutions of \eqref{fractionalPME}. Combining those results with the uniqueness of weak solutions of \eqref{fractionalPME}, we can conclude that the aforementioned limit point $\mathbb{Q}$ is actually unique, leading to the conclusion of Theorem \ref{hydlim}. Finally in Subsection \ref{sec:replem} we prove a lemma which is required in Subsections \ref{sec:charac} and \ref{sec:EE}.

\section{Heuristic arguments for the hydrodynamic equations}
\label{sec:heuristics}

In this section we present an heuristic reasoning which leads us to the integral equation \eqref{eq:int_eq} in Definition \ref{eq:dif}. We begin by assuming that the sequence $(\mathbb{Q}_{n})_{n \geq 1}$ is tight  (this fact will be proved in the next section) and we denote by $\mathbb{Q}$ a limit point. 
A simple computation based on the fact that our variables are bounded, allows showing that the limit  measure $\mathbb{Q}$ is concentrated on trajectories of measures $\pi_t(du)$ that are  absolutely continuous with respect to the Lebesgue measure, i.e. $\pi_t(du):=\rho(t,u)du$. Now we need to characterize $\rho(t,u)$ as a weak solution to the corresponding fractional equation.  According to Dynkin's formula (see Lemma $A.1.5.1$ of \cite{kipnis1998scaling}), 
\begin{equation}\label{dynkin}
	\mcb M_{t}^{n}(G) :=\langle  \pi^{n}_{t}, G_t  \rangle - \langle \pi^{n}_{0}, G_0 \rangle - \int_{0}^{t} \partial_s \langle \pi^{n}_{s},G_s \rangle\,ds - \int_{0}^{t}  n^{\gamma} \mcb L\langle \pi^{n}_{s},G_s \rangle\,ds
\end{equation}
is a martingale with respect to $\mcb {F}^n_{t} := \left\{ \sigma(\eta_s^n): s\leq t \right\}$, for every $n \geq 1$, $t \in [0,T]$ and $G \in \mcb{S}$. Since the sequence $(\mathbb Q_n)_{n\geq 1}$ is tight let $n_j$ be a subsequence such that $(\mathbb Q_{n_j})_{j\geq 1}$ weakly converges to $\mathbb Q$, which is supported on trajectories of the form $\pi_t(du)=\rho(t,u)du$. To make notation simple we assume that $n_j=n$. From this it follows that the first three terms on the right-hand side of  \eqref{dynkin} converge, as $n \rightarrow \infty$, in $L^1 (\mathbb{P}_{\mu_n} )$ to
\begin{align*} 
	\int_{\mathbb{R}} \rho_t(u) G_t (u) \, du  - \int_{\mathbb{R}} \rho_0(u)G_0(u) \, du   -  \int_0^t   \int_{\mathbb{R}} \rho_s(u)  \partial_s G_s(u) \, du \,  ds.
\end{align*}
From Definition \ref{associated profile}, we get that for any $\delta>0$,  $$\mathbb Q\Big(\Big|\int_{\mathbb{R}} [\rho_0(u) - g(u)] G_0(u) \, du \,\Big|>\delta\Big)=0.$$  Hence last display converges in $L^1 (\mathbb{P}_{\mu_n} )$ to
% replace last expression by
\begin{align*}
	\int_{\mathbb{R}} \rho_t(u) G_t(u) \, du  - \int_{\mathbb{R}} g(u) G_0(u) \, du   -  \int_0^t   \int_{\mathbb{R}} \rho_s(u)  \partial_s G_s(u) \, du \;   ds,
\end{align*}   
as $n \rightarrow \infty$. Now we focus on last term of \eqref{dynkin}, which is known in the literature as the \textit{integral term}; it describes the action of the infinitesimal generator in the empirical measure associated to the conserved quantity: the density of particles. 

The integral term  will lead us to the fractional heat equation (resp. fractional porous medium equation) for $m=1$ (resp. $2\leq m\in\mathbb N$). Indeed, by fixing $x_0 \in \mathbb{Z}$ and making $f=\eta(x_0)$ in \eqref{geninf}, we have
\begin{align}  \label{lnetaz}
	\mcb L \big( \eta(x_0) \big) = \frac{1}{2} \sum_{y } p(y-x_0) c^{(m)}_{x_0,y}(\eta) [\eta(y)-\eta(x_0)].
\end{align}
Above we used the facts that $[\eta^{y,z}(x_0)-\eta(x_0)]=0$ when $y,z  \in \mathbb{Z} - \{x_0 \}$ and that $p(y-x_0) c^{n}_{x_0,y}(\eta) = p(x_0-y) c^{n}_{y,x_0}(\eta)$. From the linearity of $\mcb {L}$ and \eqref{lnetaz}, we get
\begin{align} \label{gendynk}
	n^{\gamma} \mcb{L} ( \langle \pi_s^n, G_s \rangle )  
	= \frac{n^{\gamma}}{2n} \sum_{x,y}  G_s(\tfrac{x}{n}) p(y-x)  c^{(m)}_{x,y}(\eta_s^n)  [\eta_s^n(y)-\eta_s^n(x)]. 
\end{align}
Next we analyse \eqref{gendynk} for particular values of $m$. 
\subsection{Heuristics for $m=1$}
For $m=1$, we get from \eqref{defrates} that $c^{(m)}_{x,y}(\eta_s^n)  [\eta_s^n(y)-\eta_s^n(x)]=2[\eta_s^n(y)-\eta_s^n(x)]$, therefore the right-hand side of \eqref{gendynk} can be rewritten as
\begin{align*}
	&\frac{n^{\gamma}}{n} \sum_{x,y}  G_s(\tfrac{x}{n}) p(y-x)    \eta_s^n(y) - \frac{n^{\gamma}}{n} \sum_{x,y}  G_s(\tfrac{x}{n}) p(y-x)    \eta_s^n(x) \\
	=&  \frac{1}{n} \sum_{x,y} n^{\gamma}[G_s(\tfrac{y}{n}) -  G_s(\tfrac{x}{n})] p(y-x)  \eta_s^n(x).
\end{align*}
In last equality we used the symmetry of $p(\cdot)$. Therefore,  
\begin{align} \label{terprinc1} 
	\int_{0}^{t}  n^{\gamma} \mcb L\langle \pi^{n}_{s},G_s\rangle\,ds = \int_{0}^{t} \frac{1}{n}\sum_{x }n^{\gamma}\mcb{K}_{n}G_s(\tfrac{x}{n}) \eta_s^n(x) \, ds, 
\end{align}
where $\mcb{K}_{n}$ is defined on functions $G \in \mcb{S}$ by
\begin{equation}\label{defKn}
	\mcb{K}_{n}G_s(\tfrac{x}{n}) :=\sum_{y}\left[G_s(\tfrac{y}{n})-G_s(\tfrac{x}{n})\right]p(y-x).
\end{equation}
The term on the right-hand side of \eqref{terprinc1} can be treated with next result, which motivates the choice $n^{\gamma}$ for the choice of the  time scale. We refer the reader to Proposition A.1 of \cite{CGJ2} for its proof.
\begin{prop} \label{convdisc}
	For every $\gamma \in (0,2)$ and $G \in \mcb S$, it holds
	\begin{align*}
		\lim_{n \rightarrow \infty} \frac{1}{n}   \sum_{x }  \sup_{s \in [0,T]}\big|n^{\gamma} \mcb{K}_n G_s \left(\tfrac{x}{n} \right)  -[-(- \Delta)^{\gamma/2}   G_s]  \left(\tfrac{x}{n} \right) \big| =0.
	\end{align*}
\end{prop} 
Note that \eqref{terprinc1} can be written as $\int_0^t \langle \pi_s^n, n^{\gamma} \mcb{K}_n G_s\rangle$ and so it is written as a function of the empirical measure. 
Thanks to Proposition \ref{convdisc}, we conclude that for $m=1$, the integral term converges, as $n \rightarrow \infty$, in $L^1 (\mathbb{P}_{\mu_n} )$ to
\begin{align*} 
	\int_0^t \int_{\mathbb{R}} \rho_s(u)  [-(- \Delta)^{\gamma/2} G_s](u) \, du \, ds,
\end{align*} 
and when we collect all terms we get the integral equation in Definition \ref{eq:dif}.
\subsection{Heuristics for $m \geq 2$}
For $m \geq 2$, we will perform some algebraic manipulations in order to rewrite the integral term as the sum of two terms. The first one will be the \textit{principal term}, which will converge as $n \rightarrow \infty$ and in $L^1 (\mathbb{P}_{\mu_n} )$ to
\begin{align} \label{weakform}
	\int_0^t \int_{\mathbb{R}} \rho^m_s(u)  [-(- \Delta)^{\gamma/2} G_s](u) \, du \, ds,
\end{align}   
which is the final term to be obtained in the integral equation in Definition \ref{eq:dif}. The second one is the \textit{extra term}, that vanishes  as $n \rightarrow \infty$ and in $L^1 (\mathbb{P}_{\mu_n} )$.
\subsubsection{Heuristics for $m=2$}
We begin by treating the particular case $m=2$, in the same way as it was done in \cite{CDG}. In this case, Remark \ref{ratpor} leads to
\begin{align*}
	&[c^{(m,dif)}_{x,y}(\eta) + c^{(m,dif)}_{y,x}(\eta)  ] [\eta(y)-\eta(x)] \\
	=& [ \eta(x-1) + \eta(y+1) + \eta(y-1) + \eta(x+1) ][\eta(y) - \eta(x) ] \\
	=& [\eta(y) \eta(y-1) + \eta(y) \eta(y+1)] - [\eta(x) \eta(x-1) + \eta(x) \eta(x+1)] \\
	+&[ \eta(y) \eta(x-1)- \eta(y+1) \eta(x) ] - [ \eta(x) \eta(y-1)- \eta(x+1) \eta(y) ].  
\end{align*}
In the third line of last display, we simply collected the products of $\eta$'s that  either depend only on $x$ or depend only on $y$. And in last line, we collected all the remaining products (which depend both  on $x$ and $y$). The principal term (resp. extra term) will be produced by the terms in  the third line (resp. the last line). The extra term converges vanishes in $L^1 (\mathbb{P}_{\mu_n} )$  as $n \rightarrow \infty$, due to Proposition \ref{convext}. Indeed, combining last display with an exchange of variables, the right-hand side of \eqref{gendynk} can be rewritten as
\begin{align*}
	& \frac{1}{2n}\sum_{x,y }n^{\gamma} [ G_s(\tfrac{y}{n})  p(x-y) -   G_s(\tfrac{x}{n})  p(y-x)] [ \eta_s^n(x) \eta_s^n(x-1) + \eta_s^n(x) \eta_s^n(x+1) ]   \\
	+ & \frac{1}{2n}\sum_{x,y } n^{\gamma} [ G_s(\tfrac{y}{n})  p(x-y) -   G_s(\tfrac{x}{n})  p(y-x)] [ \eta_s^n(x) \eta_s^n(y-1) - \eta_s^n(x+1) \eta_s^n(y) ] .
\end{align*}
From the symmetry of $p(\cdot)$, last display can be rewritten as 
\begin{align*}
	&  \frac{1}{2n}\sum_{x }n^{\gamma}\mcb{K}_{n}G_s(\tfrac{x}{n}) [ \eta_s^n(x) \eta_s^n(x-1) + \eta_s^n(x) \eta_s^n(x+1) ]   \\
	+ & \frac{1}{2n}\sum_{x,y }n^{\gamma} [ G_s(\tfrac{y}{n}) - G_s(\tfrac{x}{n})] p(y-x) [ \eta_s^n(x) \eta_s^n(y-1) - \eta_s^n(x+1) \eta_s^n(y) ] . 
\end{align*}
Above, the term in first line (resp. second line) is the principal term (resp. extra term). Contrarily to what we have seen in \eqref{terprinc1}, now the principal term cannot be straightforwardly written in terms of the empirical measure (and this is the case for all $m \geq 2$ and for this reason extra arguments are needed in order to write this term as a function of the empirical measure. 
\subsubsection{Heuristics for $m=3$}
For $m=3$, we have an extra contribution of the SSEP dynamics that we added in order to have always the possibility of performing jumps of length one (in the previous case this was granted by the porous medium dynamics). In this case, from Remark \ref{ratpor},  we have 
%\begin{align*}
%&[c^{(m,dif)}_{x,y}(\eta) + c^{(m,dif)}_{y,x}(\eta)  ] [\eta(y)-\eta(x)]\\
%%=&[\eta(x-2) \eta(x-1) + \eta(x-1) \eta (y + 1 ) + \eta ( y + 1 ) \eta ( y + 2)  ][\eta(y) - \eta(x) ] \\
%%+& [ \eta(y-2) \eta(y-1) + \eta(y-1) \eta (x + 1 ) + \eta ( x + 1 ) \eta ( x + 2) ][\eta(y) - \eta(x) ] \\
%=& [\eta(y) \eta(y-1) \eta(y-2) + \eta(y) \eta(y+1) \eta(y+2)] \\
%-& [\eta(x) \eta(x-1) \eta(x-2) + \eta(x) \eta(x+1) \eta(x+2)] \\
%+&[ \eta(y) \eta(x-1) \eta(x-2) + \eta(y) \eta(y+1) \eta(x-1)\\&\quad\quad\quad\quad \quad  -  \eta(y+1) \eta(x) \eta(x-1) - \eta(y+1) \eta(y+2) \eta(x) ] \\
% - & [ \eta(x) \eta(y-1) \eta(y-2) + \eta(x) \eta(x+1) \eta(y-1) \\&
%\quad\quad\quad\quad \quad -  \eta(x+1) \eta(y) \eta(y-1) - \eta(x+1) \eta(x+2) \eta(y) ].  
%\end{align*}
\begin{align*}
	&[c^{(m,dif)}_{x,y}(\eta) + c^{(m,dif)}_{y,x}(\eta)  ] [\eta(y)-\eta(x)]\\
	%=&[\eta(x-2) \eta(x-1) + \eta(x-1) \eta (y + 1 ) + \eta ( y + 1 ) \eta ( y + 2)  ][\eta(y) - \eta(x) ] \\
	%+& [ \eta(y-2) \eta(y-1) + \eta(y-1) \eta (x + 1 ) + \eta ( x + 1 ) \eta ( x + 2) ][\eta(y) - \eta(x) ] \\
	=& [\eta(y) \eta(y-1) \eta(y-2) + \eta(y) \eta(y+1) \eta(y+2)] \\
	-& [\eta(x) \eta(x-1) \eta(x-2) + \eta(x) \eta(x+1) \eta(x+2)] \\
	+& \big[ \eta(y) \eta(x-1) \eta(x-2) + \eta(y) \eta(y+1) \eta(x-1)\\
	- &  \eta(y+1) \eta(x) \eta(x-1) - \eta(y+1) \eta(y+2) \eta(x) \big] \\
	- & \big[ \eta(x) \eta(y-1) \eta(y-2) + \eta(x) \eta(x+1) \eta(y-1) \\
	- & \eta(x+1) \eta(y) \eta(y-1) - \eta(x+1) \eta(x+2) \eta(y) \big].  
\end{align*}
In the second and third lines of last display, we simply collected the products of $\eta$'s which either depend only on $r$ or depend only on $y$. On the other hand, in the last four lines, we collected all the remaining products (which depend of both $x$ and $y$). The principal term (resp. extra term) will be produced by the terms in the second and third lines (resp. last four lines). The extra term converges in $L^1 (\mathbb{P}_{\mu_n} )$ to zero as $n \rightarrow \infty$, due to Proposition \ref{convext} below. Indeed, combining last display with an exchange of variables, the right-hand side of \eqref{gendynk} can be rewritten as
\begin{align*}
	&\int_{0}^{t} \frac{n^{\gamma}}{2n}\sum_{x,y } [ G_s(\tfrac{y}{n})  p(x-y) -   G_s(\tfrac{x}{n})  p(y-x)] \times \\
	& \times[ \eta_s^n(x) \eta_s^n(x-1) \eta_s^n(x-2) + \eta_s^n(x) \eta_s^n(x+1) \eta_s^n(x+2) ] \,  ds \\
	+ & \int_{0}^{t} \frac{n^{\gamma}}{2n}\sum_{x,y } [ G_s(\tfrac{y}{n})  p(x-y) -   G_s(\tfrac{x}{n})  p(y-x)] \times \\
	\times & \big( \eta_s^n(x)  \eta_s^n(y-1) [\eta_s^n(y-2) +  \eta_s^n(x+1)  ] - \eta_s^n(x+1)  \eta_s^n(y) [\eta_s^n(y-1) +  \eta_s^n(x+2)  ] \big) \,  ds \\
	+&  \int_{0}^{t} \frac{n^{\gamma}}{2n} \sum_{x } \big(  G_s(\tfrac{x+1}{n})  p(-1) + G_s(\tfrac{x-1}{n})  p(1) -   G_s(\tfrac{x}{n}) [ p(1) +   p(-1)] \big)\eta_s^n(x) \, ds.
\end{align*}
From the symmetry of $p(\cdot)$, last display can be rewritten as 
\begin{align*}
	& \int_{0}^{t} \frac{1}{2n} \sum_{x }n^{\gamma}\mcb{K}_{n}G_s(\tfrac{x}{n}) [ \eta_s^n(x) \eta_s^n(x-1) \eta_s^n(x-2) + \eta_s^n(x) \eta_s^n(x+1) \eta_s^n(x+2) ] \, ds   \\
	+ & \int_{0}^{t} \Big\{  \frac{1}{2n}\sum_{x,y }n^{\gamma} [ G_s(\tfrac{y}{n}) - G_s(\tfrac{x}{n})] p(y-x) \eta_s^n(x)  \eta_s^n(y-1) [\eta_s^n(y-2) +  \eta_s^n(x+1)  ] \\
	- & \frac{1}{2n}\sum_{x,y }n^{\gamma} [ G_s(\tfrac{y}{n}) - G_s(\tfrac{x}{n})]  p(y-x) \eta_s^n(x+1)  \eta_s^n(y) [\eta_s^n(y-1) +  \eta_s^n(x+2)  ]  \Big\} \, ds  \\
	+ & \int_{0}^{t}  \frac{p(1)}{2n} \sum_{x }n^{\gamma} [ G_s(\tfrac{x+1}{n}) + G_s(\tfrac{x-1}{n}) - 2 G_s(\tfrac{x}{n})  ] \eta_s^n(x) \, ds.
\end{align*}
The term in the fourth line of last display comes from the contribution of the SSEP and it vanishes as $n \rightarrow \infty$, since it is not scaled on the diffusive scale $n^2$ but on $n^\gamma$.
\subsubsection{Heuristics for the general case $m \geq 2$}
For every $m \geq 2$, from Remark \ref{ratpor}, it holds
\begin{equation} \label{algman}
	\begin{split}
		&[c^{(m,dif)}_{x,y}(\eta) + c^{(m,dif)}_{y,x}(\eta)] [\eta(y) - \eta(x)] =  B_{m}( \eta, y ) - B_{m}( \eta, x ) \\
		+& [ C_{m}( \eta, y , x  ) - C_{m}( \eta,y+1, x+1  )  ] - [ C_{m}( \eta,x, y  ) - C_{m}( \eta,x+1, y+1  ) ], 
	\end{split}
\end{equation}
where $B_{m}: \Omega \times \mathbb{Z} \rightarrow [0,2]$ and $C_{m} : \Omega \times \mathbb{Z} \times \mathbb{Z} \rightarrow [0,m-1]$ are given by
\begin{align}
	& B_{m}( \eta, z ):= \prod_{i=0}^{m-1}   \eta(z-i ) + \prod_{i=0}^{m-1}   \eta(z+i ); \label{defBm} \\
	& C_{m}( \eta, z, w  ):= \sum_{k=1}^{m-1} \prod_{j=0}^{k-1} \eta( z + j   ) \times \prod_{i=1}^{m-k} \eta( w - i  ). \label{defCm}
\end{align}
Combining \eqref{algman} with an exchange of variables, the right-hand side of \eqref{gendynk} can be rewritten as
\begin{align*}
	& \frac{n^{\gamma}}{2n}\sum_{x,y }  [ G_s(\tfrac{y}{n})  p(x-y) -   G_s(\tfrac{x}{n})  p(y-x)] B_{m}( \eta_s^n, x ) \\
	+ &  \frac{n^{\gamma}}{2n}\sum_{x,y }  [ G_s(\tfrac{y}{n})  p(x-y) -   G_s(\tfrac{x}{n})  p(y-x)] [ C_{m}( \eta_s^n,x, y  ) - C_{m}( \eta_s^n,x+1, y+1  )   ]\\
	+&   \frac{n^{\gamma}}{2n} \sum_{x }  \big( [ G_s(\tfrac{x+1}{n})  p(-1) -   G_s(\tfrac{x}{n})  p(1)] + [ G_s(\tfrac{x-1}{n})  p(1) -   G_s(\tfrac{x}{n})  p(-1)] \big)\eta_s^n(x) .
\end{align*}
For every $G \in \mcb S$, define $\nabla G^n_s$ by $\nabla G^n_s(\tfrac{x}{n})=n[ G_s(\tfrac{x}{n})-G_s(\tfrac{x-1}{n})]$.  Then, from the symmetry of $p(\cdot)$, the  integral term can be rewritten as 
\begin{align}
	& \int_{0}^{t} \frac{1}{2n} \sum_{x }n^{\gamma}\mcb{K}_{n}G_s(\tfrac{x}{n}) B_m( \eta_s^n, x ) \, ds \label{terprinc} \\
	+ & \int_{0}^{t} \frac{1}{2n^2}\sum_{x,y } n^{\gamma} [ \nabla G^n_s(\tfrac{y}{n}) - \nabla G^n_s(\tfrac{x}{n}) ] p(y-x)  C_{m}( \eta_s^n,x, y  ) \, ds \label{terext1} \\
	+ & \mathbbm{1}_{ \{ m \geq 3 \} }  \int_{0}^{t} \frac{p(1)}{2n} \sum_{x }n^{\gamma} [ G_s(\tfrac{x+1}{n}) + G_s(\tfrac{x-1}{n}) - 2 G_s(\tfrac{x}{n})  ] \eta_s^n(x) \, ds. \label{terext2}
\end{align}
In order to obtain \eqref{terext1}, we performed the change of variables $(z,w)=(x+1,y+1)$. Thus, the principal term (resp. the extra term) is given by \eqref{terprinc} (resp. the sum of \eqref{terext1} and \eqref{terext2}). In particular, the extra term can be rewritten as $\int_0^t \mcb{R}_{n}^G(s)ds$, where
\begin{equation}\label{defR}
	\begin{split}
		\mcb{R}_{n}^G(s) :=& \frac{1}{2n^2}\sum_{x,y } n^{\gamma} [ \nabla G^n_s(\tfrac{y}{n}) - \nabla G^n_s(\tfrac{x}{n}) ] p(y-x) C_{m}( \eta_s^n,x, y  )  \\
		+ & \mathbbm{1}_{ \{ m \geq 3 \} }  \frac{p(1)}{2n} \sum_{x }n^{\gamma} [ G_s(\tfrac{x+1}{n}) + G_s(\tfrac{x-1}{n}) - 2 G_s(\tfrac{x}{n})  ] \eta_s^n(x).
	\end{split}  
\end{equation}
Combining \eqref{defCm} with the fact that the variables are bounded (i.e. $|\eta_s^n(\cdot)| \leq 1$), we get $| C_{m}( \cdot , \cdot , \cdot  ) | \leq m - 1$. Thus, we have that $\sup_{s \in [0,T]} |\mcb{R}_{n}^G(s)| \leq Y_1^G + Y_2^G$ for every $G \in \mcb S$, where $Y_1^G$ and $Y_2^G$ are given by
\begin{align*}
	Y_1^{G,n}:=& \frac{m-1}{n^2}   \sum_{x,y }  \sup_{s \in [0,T]} n^{\gamma} | \nabla G^n_s(\tfrac{y}{n}) - \nabla G^n_s(\tfrac{x}{n}) | p(y-x);  \\
	Y_2^{G,n}:=& \frac{1}{n} \sum_{x } \sup_{s \in [0,T]}  n^{\gamma} \big| G_s(\tfrac{x+1}{n}) + G_s(\tfrac{x-1}{n}) - 2 G_s(\tfrac{x}{n})\big|.
\end{align*}
Next we claim that  for all $G \in \mcb S$,  it holds $\lim_{n \rightarrow \infty} Y_1^{G,n} + Y_2^{G,n} = 0.$
In order to treat $ Y_1^{G,n}$ we apply Proposition \ref{convext} below. We refer the interested reader in its proof to Proposition 3.1 of \cite{CDG}.
\begin{prop} \label{convext}
	For  $\gamma \in (0,2)$, let
	$\delta_{\gamma}=
	\frac 12\textbf{1}_{\gamma=1}+\textbf{1}_{\gamma \in (1,2)}$.
	Then {for any $G \in \mcb S$} 
	\begin{align*}
		& Y_1^{G,n} \lesssim \max \Big\{ n^{\gamma-2} , n^{-1} , n^{\gamma-1-\delta_{\gamma}}\Big\}. 
	\end{align*}
	In particular, since $ \delta_{\gamma} > \gamma - 1$ and $\gamma < 2$, last display vanishes as $n \rightarrow \infty$.
\end{prop}
In order to treat $Y_2^{G,n}$, we apply a second-order Taylor expansion of $G_s$ around $x/n$ to conclude that
\begin{align*}
	\lim_{n \rightarrow \infty} \frac{1}{n} \sum_{x } \sup_{s \in [0,T]} \big| n^{2} [ G_s(\tfrac{x+1}{n}) + G_s(\tfrac{x-1}{n}) - 2 G_s(\tfrac{x}{n})] - \Delta G_s (\tfrac{x}{n})   \big| =0.
\end{align*}
Combining this with the fact that $\gamma < 2$, we have that $\lim_{n \rightarrow \infty} Y_2^{G,n}=0$ for every $G \in \mcb S$ and this proves the claim. In particular, it holds 
\begin{equation} \label{boundL1RnG}
	\lim_{n \rightarrow \infty} \mathbb{E}_{\mu_n} \Big[ \sup_{t \in [0,T]} \Big| \int_{0}^{t}  \mcb {R}_{n}^G(s) \, ds \,   \Big|  \Big] =0.
\end{equation}
%Note that in particular, the sum of \eqref{terext1} and \eqref{terext2} converges, as $n \rightarrow \infty$, in $L^1 (\mathbb{P}_{\mu_n} )$ to zero. 
Finally, we treat the display in \eqref{terprinc}; from \eqref{defBm}, it can be rewritten as
\begin{align*}
	\int_{0}^{t} \frac{1}{2n} \sum_{x }n^{\gamma} \big[\mcb{K}_{n}G_s(\tfrac{x}{n}) + \mcb{K}_{n}G_s(\tfrac{x+m-1}{n}) \big] \prod_{i=0}^{m-1} \eta^n_{s}(x+i ) \, ds.
\end{align*}
Due to the products of $\eta_s^n$'s, it is not straightforward to close the integral term in terms of the empirical measure; in order to do so, we apply Lemma \ref{globrep} below. Before stating this lemma, we introduce a bit of notation. For every $\ell \geq 1$ in $\mathbb{Z}$ and $x \in \mathbb{Z}$, we denote the empirical average in a box of side $\ell$ \textit{at the right} of $x$ by
\begin{equation}  \label{medempright}
	\overrightarrow{\eta}^{\ell}(x):= \frac{1}{\ell} \sum_{i=1}^{\ell} \eta (x + i).
\end{equation}
\begin{lem} \textbf{(Replacement Lemma)} \label{globrep}
	Let $(\Phi_{n})_{n \geq 1} : [0,T] \times \mathbb{R}  \rightarrow \mathbb{R}$ be a sequence of functions satisfying 
	\begin{align} \label{boundrep}
		\frac{1}{n} \sum_{x} \sup_{s \in [0,T]} | \Phi_{n}(s, \tfrac{x}{n} ) | \leq M_1 \quad\textrm{and}\quad\| \Phi_{n} \|_{\infty}:= \sup_{(s,u) \in [0,T] \times \mathbb{R}} \hspace{-0.5 cm} | \Phi_n (s,u)| \leq M_2,
	\end{align}
	for all $n \geq 1$ and for some positive constants $M_1$, $M_2$. Under Assumption (A), for any $m \geq 2$ and every $t \in [0,T]$, it holds
	\begin{equation*}% \label{rl1}
		\varlimsup_{\varepsilon \rightarrow 0^{+}}\varlimsup_{n \rightarrow \infty} \mathbb{E}_{\mu_n} \Big[  \Big| \int_{0}^{t}\frac{ 1  }{n}\sum_{x}  \Phi_{n} (s, \tfrac{x}{n} ) \Big\{ \prod_{i=0}^{m-1} \eta^n_{s}(x+i ) - \prod_{i=0}^{m-1} \overrightarrow{\eta}_{s}^{\varepsilon n} \big(x+  i \varepsilon n) \Big \} \, ds \, \Big| \Big] =0.
	\end{equation*}
\end{lem}
In order to apply Lemma \ref{globrep}, we need to ensure that \eqref{boundrep} is satisfied. This is done by applying a classical result regarding the fractional Laplacian, Proposition \ref{propL1} below. We refer the interested reader in its proof to Proposition 3.1 of \cite{CDG} for more details.
\begin{prop} \label{propL1}
	For every $G \in \mcb{S}$, we have
	\begin{equation} \label{L1frac}
		\sup_{n \geq 1}	\frac{1}{n} \sum_{x} \sup_{s \in [0,T]} \big|  [ -(-\Delta)^{\gamma/2}G_s]( \tfrac{x}{n})\big| < \infty.
	\end{equation}
\end{prop}
Thus, by combining Propositions \ref{convdisc} and \ref{propL1} with Lemma \ref{globrep}, we have that \eqref{terprinc} converges, as $n \rightarrow \infty$, to \eqref{weakform}, in $L^1(\mathbb{P}_{\mu_n})$, and we get the integral equation in Definition \ref{eq:dif}.
\section{Proof of Theorem \ref{hydlim}} \label{sec:main_theo}
\subsection{Tightness}
~\label{sec:tight}
Our goal is to prove that  $(\mathbb{Q}_{n})_{n \geq 1}$ is tight with respect to the Skorokhod topology of $\mcb {D}([0,T],  \mcb{M}^{+})$. Following Propositions 4.1.6 and 4.1.7 of \cite{kipnis1998scaling}, it is enough to show that
\begin{equation}\label{tight}
	\lim_{\omega \rightarrow 0} \varlimsup_{n \rightarrow \infty} \sup_{\tau  \in \mcb{T}_{T}, \tilde{t} \leq \omega}\mathbb{P}_{\mu_{n}} \Big(  \eta_{\cdot}^n \in \mcb{D}\left([0,T], \Omega \right): \big| \langle \pi_{\tau + \tilde{t}}^{n},G \rangle - \langle \pi_{\tau}^{n},G \rangle \big| >\varepsilon \Big)=0,
\end{equation}
for every $G \in \mcb S$ and every $\varepsilon > 0$. Above, $\mcb{T}_{T}$ is the set of stopping times bounded by $T$, therefore $\tau + \tilde{t}$ must be read as $\min\{\tau + \tilde{t}, T\}$. % \textcolor{blue}{Since $C_c^{\infty}(\mathbb{R}) \subsetneq \mcb S$, we can present the proof of tightness for $G \in \mcb S$.}  
From \eqref{dynkin}, Markov's and Chebyshev's inequalities, \eqref{tight} is bounded from above by
\begin{align}
	& \lim_{\omega \rightarrow 0} \varlimsup_{n \rightarrow \infty} \sup_{\tau  \in \mcb{T}_{T}, \tilde{t} \leq \omega} \frac{2}{\varepsilon} \mathbb{E}_{\mu_{n}} \Big[ \Big| \int_{\tau }^{\tau + \tilde{t}}  n^{\gamma} \mcb L \langle \pi_{r}^{n},G_r \rangle \, dr \, \Big| \Big] \label{limtight1} \\
	+ & \lim_{\omega \rightarrow 0} \varlimsup_{n \rightarrow \infty} \sup_{\tau  \in \mcb{T}_{T}, \tilde{t} \leq \omega} \frac{4}{\varepsilon^2} \mathbb{E}_{\mu_{n}}\Big[  |\mcb M_{\tau + \tilde{t}}^{n}(G) - \mcb M_{\tau}^{n}(G) \big|^2 \Big]. \label{limtight2}
\end{align}
First it is enough to prove that both limits in last display vanish as $n \rightarrow \infty$, for every $G \in \mcb S$ fixed. 
We begin with the former one. Combining \eqref{terprinc} and \eqref{defR} with Proposition \ref{convdisc} and the fact that  $\lim_{n \rightarrow \infty} Y_1^{G,n} + Y_2^{G,n} = 0$, we can conclude, in particular, that for every $G \in \mcb S$, there exists $C(G)$ such that  $\sup_{s \in [0,T]}  |n^{\gamma} \mcb L \langle \pi^{n}_{s}, G_s  \rangle | \leq C(G)$, therefore the limit in \eqref{limtight1} is equal to zero. From Lemma $A1.5.1$ of \cite{kipnis1998scaling}, the expectation in \eqref{limtight2} is equal to
\begin{align*}
	\mathbb{E}_{\mu_n}\Bigg[ \int_{\tau}^{\tau + \tilde{t}} n^{\gamma}\big[ \mcb L \langle \pi_{s}^{n},G_s\rangle^{2} -2\langle \pi_{s}^{n}, G_s\rangle \mcb L \langle \pi_{s}^{n}, G_s \rangle \big] \, ds \Bigg],
\end{align*}
therefore to finish the proof it is enough to apply the next result.
\begin{prop} \label{boundquad}
	Let $G \in \mcb{S}$. Then
	\begin{equation} \label{boundquad1}
		\sup_{s \in [0,T]}\Big| n^{\gamma}\left(\mcb L \langle \pi_{s}^{n},G _s \rangle^{2} -2\langle \pi_{s}^{n}, G_s \rangle  \mcb L \langle \pi_{s}^{n}, G_s\rangle \right)\Big| \lesssim \max\{ n^{\gamma-2}, n^{-1} \}. 
	\end{equation}
	In particular,  
	\begin{align} 
		\forall \delta_1 >0, \quad \lim_{n \rightarrow \infty} \mathbb{P}_{\mu_n} \Big( \sup_{t \in [0,T]} | \mcb M_{t}^{n}(G) | > \delta_1 \Big) =0. \label{condger2pr3}
	\end{align}
\end{prop}
\begin{proof}%[Proof of Proposition \ref{boundquad}]
	We begin by showing \eqref{condger2pr3} from \eqref{boundquad1}. Since $\left(  |\mcb M_{t}^{n}(G)| \right)_{t \geq 0}$ is a non-negative submartingale, Doob's inequality leads to
	\begin{align*}
		& \mathbb{P}_{\mu_n}  \Big( \sup_{s \in [0,T]} | \mcb  M_{s}^{n}(G)   | > \delta_1  \Big)  \leq  \dfrac{1}{\delta_1} \left( \mathbb{E}_{\mu_n} \left[ \big( \mcb  M_{T}^{n}(G)   \big)^2 \right] \right)^{ 1 / 2} \\
		=&  \frac{1}{\delta_1}   \Big( \int_0^T \sup_{s \in [0,T]} |n^{\gamma} \left(\mcb L[\langle \pi_{s}^{n},G_s  \rangle]^{2}-2\langle \pi_{s}^{n},G_s  \rangle \mcb L \langle \pi_{s}^{n},G_s  \rangle\right)| \, ds \,  \Big)^{ 1 / 2},
	\end{align*} 
	which goes to zero as $n \rightarrow \infty$, for every $\delta_1 >0$ fixed, due to \eqref{boundquad1}. Now we go to the second part of the proof. After  some algebraic manipulations, we rewrite the expression on the left-hand side of \eqref{boundquad1} as
	\begin{align*}
		\sup_{s \in [0,T]} \frac{n^{\gamma}}{4n^{2}}\sum_{x,y } [ G_s\left(\tfrac{y}{n}\right)-G_s\left(\tfrac{x}{n}\right) ]^{2} p(x-y)  c^{(m)}_{x,y}(\eta_s^n) [\eta_{s}^n(x)-\eta_{s}^n(y) ]^{2} .
	\end{align*}
	Due to \eqref{defrates}, we get $|c^{(m)}_{x,y}(\eta_s^n)| \leq 2m +1$. Combining this with with Proposition A.10 of \cite{CGJ2} and with the fact that $[\eta_{s}^n(x)-\eta_{s}^n(y) ]^{2} \leq 1$, the proof ends.
\end{proof}

\subsection{Characterization of limit points}\label{sec:charac}

From the results of Section \ref{sec:tight}, we know  that $(\mathbb{Q}_n)_{n \geq 1}$ has at least one limit point $\mathbb{Q}$. From Section 4.1 of \cite{kipnis1998scaling}, since every site has at most one particle, any limit point $\mathbb{Q}$ is concentrated on trajectories of measures that are absolutely continuous with respect to the Lebesgue measure, i.e.
\begin{equation} \label{defdens}
	\pi_t(du)=\rho(t,u) \, du,
\end{equation}
for (almost) every $t$ on $[0,T]$ with $0 \leq \rho \leq 1$ on $[0,T] \times \mathbb{R}$. In this subsection we will prove that $\mathbb{Q}$ is concentrated on trajectories such that $\rho$ satisfies the first condition of weak solutions in Definition \ref{eq:dif}. 
%\textcolor{red}{
	%In Lemma \ref{globrep}, we do not have a supremum over $t \in [0,T]$ inside $\mathbb{E}_{\mu_n}$. Nevertheless, making use of next result, we can insert the supremum over $t \in [0,T]$ inside the probability $P_{\mu_n}$, as it is stated in \eqref{rpljustcon}. Next lemma was stated and proved in \cite{dismestotav}, and we state it here for the convenience of the reader.
	%\begin{lemma} \label{justrpl}
	%Assume that there exists a family $\mathcal F$ of functions $F_{n,\varepsilon}:  [0,T] \times \mcb D( [0,T] , \Omega ) \rightarrow \mathbb{R}$ satisfying 
	%\begin{align*} % \label{boundjustrpl}
	%\sup_{ \varepsilon \in (0,1), n \geq 1, s \in [0,T], \eta_{\cdot} \in \mcb D( [0,T] , \Omega ) } |F_{n,\varepsilon}(s, \eta_{\cdot} )| \leq M < \infty.
	%\end{align*}
	%Above, the interval for $(0,1)$ for $\varepsilon$ is arbitrary. We also assume that
	%\begin{align*} %\label{rpljust}
	%\varlimsup_{\varepsilon \rightarrow 0^{+}} \varlimsup_{n \rightarrow \infty} \mathbb{E}_{\mu_n} \Big[ \Big| \int_0^t  F_{n,\varepsilon}(s, \eta_{\cdot} ) ds \Big|\Big] =0, \forall t \in [0,T].
	%\end{align*}
	%Then 
	%\begin{align} \label{rpljustcon}
	%\varlimsup_{\varepsilon \rightarrow 0^{+}} \varlimsup_{n \rightarrow \infty} \mathbb{P}_{\mu_n} \Big( \sup_{t \in [0,T]} \Big| \int_0^t  F_{n,\varepsilon}(s, \eta_{\cdot} ) ds \Big| > \delta \Big) =0, \forall \delta > 0.
	%\end{align}
	%\end{lemma}
	%}
\begin{prop} % \label{intequ}
	If $\mathbb{Q}$ is a limit point of $(\mathbb{Q}_{n})_{n \geq 1}$ then
	\begin{equation*}
		\mathbb{Q}\Big(  \pi_{\cdot} \in \mcb{D} \left([0,T], \mcb{M}^{+}\right): \forall t\in[0,T], \forall G \in \mcb{S}, \quad F(t,\rho,G,g)=0  \Big)=1,
	\end{equation*} 
	where $F(t,\rho,G,g)$ is given in \eqref{eq:int_eq}.
\end{prop}
Last proposition identifies the profile $\rho(t,u)$ as a solution to the PDE. 
\begin{proof}
	It is enough to verify that for any $\delta>0$ and any $G \in \mcb{S}$,
	\begin{equation}\label{charact1}
		\mathbb{Q} \Big( \pi_{\cdot} \in \mcb{D} \left([0,T], \mcb{M}^{+}\right): \sup_{t \in [0,T]} |F(t,\rho,G,g)|> \delta \, \Big)=0.
	\end{equation}
	In order to simplify the notation, we will omit $\pi_{\cdot}$ from the sets where we are looking at. From the definition of $F$ in Definition \ref{eq:dif}, last probability is bounded by 
	\begin{align}
		\mathbb Q\Big( \sup_{t \in [0,T]}  \Big| & \int_{\mathbb{R}} [ \rho_t(u) G_t(u) - \rho_0(u) G_0(u)] \, du \nonumber \\
		- &  \int_0^t \int_{\mathbb{R}} \rho_s(u)  \big[ \big( \partial_s  +  [ -(- \Delta)^{ \gamma / 2}  ] \big) G_s(u)  \big]  \,  du \, ds \,   \Big| > \frac{\delta}{2} \Big) \label{f1term1sdif} \\
		\mathbb{Q} \Big(  \Big| & \int_{\mathbb{R}} \left[ \rho_0(u) -  g(u) \right] G_0(u) \, du \, \Big| > \dfrac{\delta}{2}  \Big). \label{f1term2sdif} 
	\end{align}
	The term in \eqref{f1term2sdif} is equal to zero, since  $\mathbb{Q}$ is a limit point of $(\mathbb{Q}_{n})_{n \geq 1}$ and $\mathbb{Q}_{n}$ is induced by $\mathbb{P}_{\mu_n}$, which is induced by $\mu_n$, which is associated with $g$, see Definition \ref{associated profile}. Then we are done if we can prove that the term in \eqref{f1term1sdif} is also equal to zero. 
	
	The proof for $m=1$ makes use of the same arguments applied for showing Proposition 4.2 in \cite{CGJ2}, therefore we omit it. The proof for $m \geq 2$ is analogous to the proof of Proposition 5.1 in \cite{CDG}, therefore we omit many steps. For $\varepsilon>0$ and $u \in \mathbb{R}$ fixed, we define $\overrightarrow{i_{\varepsilon}^{ u}}$ by 
	\begin{align*}
		\forall v \in \mathbb{R}, \quad \overrightarrow{i_{\varepsilon}^{ u }}(v) := \frac{1}{\varepsilon} \mathbbm{1}_{ (u, u+\varepsilon] }(v).
	\end{align*}
	Putting this together with \eqref{defdens}, we get 
	$
	\langle \pi_s, \overrightarrow{i_{\varepsilon}^{ u } }\rangle = \langle \rho_s, \overrightarrow{i_{\varepsilon}^{ u } }\rangle = \frac{1}{\varepsilon} \int_{u}^{u+\varepsilon} \rho(s,v) dv,
	$
	and since  $\rho \in [0,1]$ and Lebesgue's Differentiation Theorem, we conclude that
	\begin{equation*} % \label{aproxid}
		\lim_{\varepsilon \rightarrow 0^{+}} \Bigg|\, \prod_{j=0}^{m-1} \langle \pi_s, \overrightarrow{i_{\varepsilon}^{ u + j \varepsilon } }\rangle - \rho^m(s,u) \, \Bigg| = 0,
	\end{equation*}
	for almost every $u \in \mathbb{R}$. Therefore we are done if we can prove that
	\begin{equation}  \label{charact22}
		\begin{split}
			\varlimsup_{\varepsilon \rightarrow 0^{+}} \mathbb{Q} \Big( \sup_{t \in [0,T] } \Big| & \langle \pi_t, G_t \rangle -\langle \pi_0, G_0 \rangle - \int_0^t \langle \pi_s, \partial_s G_s \rangle \, ds \\
			-& \int_0^t \int_{\mathbb{R}} \prod_{j=0}^{m-1} \langle \pi_s, \overrightarrow{i_{\varepsilon}^{ u + j \varepsilon } }\rangle [-(-\Delta)^{\gamma/2}G_s](u) \, du \, ds \, \Big| > \frac{\delta}{3} \Big).
		\end{split}
	\end{equation}
	We observe that the functions $\overrightarrow{i_{\varepsilon}^{ u + j \varepsilon } }$ are not continuous; nevertheless, those functions can be approximated by continuous functions in such a way that the error vanishes as $\varepsilon \rightarrow 0^{+}$. Thus, by choosing convenient approximations and applying Portmanteau's Theorem (see Theorem 2.1 in \cite{billingsley}) in the same way as in the proof of Proposition 5.1 in \cite{CDG}, the display in \eqref{charact22} is bounded from above by
	\begin{equation*} %  \label{charact9}
		\begin{split}
			\varlimsup_{\varepsilon \rightarrow 0^{+}} & \varliminf_{n \rightarrow \infty} \,\,  \mathbb{P}_{\mu_n} \Big( \sup_{t \in [0,T]} \Big|  \int_{0}^{t} \frac{n^{\gamma}}{2n}\sum_{x }\mcb{K}_{n}G_s( \tfrac{x}{n}) \Big[ \prod_{j=0}^{m-1} \eta_s^n(x-j) +  \prod_{j=0}^{m-1} \eta_s^n(x+j)  \Big] \, ds \nonumber \\
			+& \mcb M_t^n(G) + \int_{0}^{t}  \mcb {R}_{n}^G(s) ds  - \int_{0}^{t} \int_{\mathbb{R}} [ -(-\Delta)^{\gamma/2}G_s](u) \prod_{j=0}^{m-1} \langle \pi_s^n, \overrightarrow{i_{\varepsilon}^{ u + j \varepsilon } }\rangle   \, du \, ds \, \Big| > \frac{\delta}{4}\Big).
		\end{split}
	\end{equation*}
	Above we applied \eqref{terprinc} and \eqref{defR}, besides summing and subtracting $\int_{0}^{t}n^{\gamma} \mcb L \langle \pi_{s}^{n}, G_s\rangle\, ds$. Since the error from changing the integral in the space variable by its Riemann sum is or order $n^{-1}$, it is enough to prove that
	\begin{equation} \label{charact91}
		\lim_{n \rightarrow \infty} \,\,  \mathbb{P}_{\mu_n} \Big( \sup_{t \in [0,T]} \Big|\mcb M_t^n(G) + \int_{0}^{t}  \mcb {R}_{n}^G(s) \, ds \,   \Big| > \frac{\delta}{8 }\Big) =0
	\end{equation}
	and
	\begin{equation*} % \label{char95}
		\begin{split}
			\varlimsup_{\varepsilon \rightarrow 0^{+}} \varliminf_{n \rightarrow \infty} \,\, \mathbb{P}_{\mu_n}\Big( \sup_{t \in [0,T]} \Big| \int_{0}^{t} \Big\{& \frac{1}{n} \sum_{x}   n^{\gamma} [ \mcb{K}_n G_s (\tfrac{x+m-1}{n}) + \mcb{K}_n G_s (\tfrac{x}{n}) ]   \prod_{j=0}^{m-1} \eta_s^n(x+j) \\
			-& \frac{1}{n} \sum_{x}  [ -(-\Delta)^{\gamma/2}G_s](\tfrac{x+m-1}{n}) \prod_{j=0}^{m-1} \langle \pi_s^n, \overrightarrow{i_{\varepsilon}^{ \frac{x}{n} + j \varepsilon } }\rangle \\
			-& \frac{1}{n} \sum_{x}  [ -(-\Delta)^{\gamma/2}G_s](\tfrac{x}{n}) \prod_{j=0}^{m-1} \langle \pi_s^n, \overrightarrow{i_{\varepsilon}^{ \frac{x}{n} + j \varepsilon } }\rangle       \Big\} \, ds \, \Big| > \frac{\delta}{8}\Big) = 0.
		\end{split}
	\end{equation*}
	By combining \eqref{condger2pr3} with \eqref{boundL1RnG} and Markov's inequality, we get \eqref{charact91}. 
	Here and in what follows we interpret $\varepsilon n$ as $\lfloor \varepsilon n\rfloor$. Putting together the estimate $$\Big|  \prod_{j=0}^{m-1} \langle \pi_s^n, \overrightarrow{i_{\varepsilon}^{ \frac{x}{n} + j \varepsilon } }\rangle - \prod_{j=0}^{m-1} \overrightarrow{\eta}_s^{\varepsilon n}(x + j \varepsilon n) \Big| \lesssim (\varepsilon n)^{-1}$$ with  \eqref{L1frac} and the fact that $[ -(-\Delta)^{\gamma/2}G_s] \in L^1(\mathbb{R})$, we are done if we can show that
	\begin{align*} 
		&\varlimsup_{\varepsilon  \rightarrow 0^{+}} \varliminf_{n \rightarrow \infty} \,\, \mathbb{P}_{\mu_n}\Big( \sup_{t \in [0,T]} \Big| \int_{0}^{t} \Big\{ \frac{1}{n} \sum_{x}   n^{\gamma} [ \mcb{K}_n G_s (\tfrac{x+m-1}{n}) + \mcb{K}_n G_s (\tfrac{x}{n}) ]   \prod_{j=0}^{m-1} \eta_s^n(x+j)  \nonumber \\
		-& \frac{1}{n} \sum_{x} \big[  [ -(-\Delta)^{\gamma/2}G_s](\tfrac{x+m-1}{n}) + [ -(-\Delta)^{\gamma/2}G_s](\tfrac{x}{n}) \big]  \prod_{j=0}^{m-1} \overrightarrow{\eta}_s^{\varepsilon n}(x + j \varepsilon n)  \Big\} \, ds \,\Big| > \frac{\delta}{8}\Big).
	\end{align*}
	is equal to zero.
	Last display is bounded by the sum of the next three terms:
	\begin{align*} 
		\varlimsup_{\varepsilon \rightarrow 0^{+}} \varliminf_{n \rightarrow \infty} \,\, \mathbb{P}_{\mu_n} \Big( \sup_{t \in [0,T]} \Big| \int_{0}^{t} \frac{1}{n} \sum_{x}  \big| \Big(n^{\gamma} \mcb{K}_n G_s -    [ -(-\Delta)^{\gamma/2}G_s]\Big)(\tfrac{x+m-1}{n})   \big| \,  ds \, \Big| > \frac{\delta}{24} \Big),
	\end{align*}
	\begin{align*}
		\varlimsup_{\varepsilon \rightarrow 0^{+}} \varliminf_{n \rightarrow \infty} \,\, \mathbb{P}_{\mu_n} \Big( \sup_{t \in [0,T]} \Big| \int_{0}^{t} \frac{1}{n} \sum_{x}  \big| n^{\gamma} \mcb{K}_n G_s (\tfrac{x}{n})-    [ -(-\Delta)^{\gamma/2}G_s](\tfrac{x}{n})   \big| \,  ds \, \Big| > \frac{\delta}{24} \Big),
	\end{align*}
	\begin{align*} 
		\varlimsup_{\varepsilon \rightarrow 0^{+}} \varliminf_{n \rightarrow \infty} \,\, \mathbb{P}_{\mu_n} \Big( \sup_{t \in [0,T]} \Big| \int_{0}^{t} &\frac{1}{n} \sum_{x} \big( [ -(-\Delta)^{\gamma/2}G_s](\tfrac{x+m-1}{n}) + [ -(-\Delta)^{\gamma/2}G_s](\tfrac{x}{n}) \big)\nonumber \\
		\times & \Big[ \prod_{j=0}^{m-1} \eta_s^n(x+j)    - \prod_{j=0}^{m-1} \overrightarrow{\eta}_s^{\varepsilon n}(x + j \varepsilon n)  \Big] \, ds \, \Big| > \frac{\delta}{24} \Big). 
	\end{align*}
	In the first and second terms of last display, we used the fact that $|\eta_s^n (\cdot)| \leq 1$, for every $s \in [0,T]$. Combining Markov's inequality with Proposition \ref{convdisc},  we conclude that  they are equal to zero. Finally, since $[ -(-\Delta)^{\gamma/2}G(s, \cdot)] \in L^1(\mathbb{R}) \cap L^{\infty}(\mathbb{R})$ (due to Proposition \ref{propL1}), from Lemma \ref{globrep}, the last term in last display is equal to zero. This ends the proof.
\end{proof}

\subsection{Energy estimates}
\label{sec:EE}
In this subsection we  prove that $\mathbb{Q}$ is concentrated on trajectories of measures whose density  $\rho$ satisfies the second condition of weak solutions in Definition \ref{eq:dif}. We do so by deriving some \textit{energy estimates}, in the sense that $\mathbb{Q}$ is concentrated in trajectories of measures such that a transformation of the density $\rho$ given in \eqref{defdens} belongs to some Hilbert space. Here and in what follows, $b$ is the constant given in Assumption (A). We do not prove the next result but we refer the reader to \cite{casodif} for a proof. In what follows, we denote $Y:=L^2([0,T] \times \mathbb{R})$ and $\| \cdot \|_{Y}$ denotes the usual norm in $Y$.
\begin{prop} \label{propestenergstat}
	Under Assumption (A), we have $\mathbb{E}_{\mathbb{Q}}  \big[ \| \rho - b \|_{Y} \big] < \infty$.
\end{prop}

Observe that {since $b \in (0,1)$, we get}
\begin{align*}
	|\lambda^m -b^m| = |\lambda-b| \Big| \sum_{i=0}^{m-1} \lambda^i b^{m-1-i} \Big| \leq  |\lambda-b| \sum_{i=0}^{\infty} b^{i} = \frac{|\lambda-b|}{1-b},
\end{align*}
for every $\lambda \in [0,1]$ and every $m \geq 1$. Therefore, the next result is a direct consequence of the previous result. 
\begin{cor} \label{corestenergstat}
	Under Assumption (A),  there exists $K_2 >0$ such that
	for all $ m \geq 1$ $\mathbb{E}_{\mathbb{Q}}  \big[ \| \rho^m - b^m \|_{Y} \big] \leq K_2.$
\end{cor}

Now we characterize a limit point $\mathbb{Q}$ in a more detailed way. More precisely, we prove that a transformation of the density $\rho$ given in \eqref{defdens} belongs to some fractional Sobolev space. 

For the remainder of this subsection, we denote $N:=C_c^{0,0,0} ( [0,T] \times \mathbb{R}^2)$. Moreover, for every $\epsilon \in (0,1]$, we define $O(\epsilon)$ and $Y(\epsilon)$ by
\begin{equation} \label{defOYeps}
	O(\epsilon):=\{(v, v) \in \mathbb{R}^2: |u-v| \geq \epsilon \}; \quad Y(\epsilon):=L^2 ( [0,T] \times \mathbb{R}^2, dt \otimes d \mu_{\epsilon } ),
\end{equation}
where  $\mu_{ \epsilon }$ is the measure 
$
(u,v) \in  \mathbb{R}^2 \rightarrow \mathbbm{1}_{ \{|u-v| \geq \epsilon \} }|u-v|^{-1-\gamma}.
$
Denote the norm associated to the Hilbert space $Y(\epsilon)$ by $\| \cdot \|_{Y(\varepsilon)}$. Now we state a preliminary lemma that will be useful in what follows. % this result is a direct consequence of the density of $(N, \| \cdot \|_{Y})$ in $(Y,\| \cdot \|_{Y})$ and of the fact that $(Y,\| \cdot \|_{Y})$ is a separable metric space.  
\begin{lem} \label{lemdenseps}
	For every $\epsilon >0$, the metric space $(N, \| \cdot \|_{Y(\epsilon)})$ is dense in the metric space $\big( Y(\epsilon), \| \cdot \|_{Y(\epsilon)} \big)$ and the metric space $\big( Y(\epsilon), \| \cdot \|_{Y(\epsilon)} \big)$ is separable.
\end{lem}
We do not present the proof of last lemma since it is classical.
Now, for $\epsilon \in (0,1]$, we define the (random) linear functional $\ell^{(m)} _{\rho, \epsilon}: N \rightarrow \mathbb{R}$ by
\begin{align*}
	\forall H \in N, \quad \ell^{(m)} _{\rho, \epsilon}(H) : =  \int_0^T  \iint_{ O( \epsilon)}  \frac{ [ \rho^m(t,v) - \rho^m(t,u) ]H_t(u,v)}{|u - v|^{1 + \gamma}} \, du \, dv \, dt.
\end{align*} 
Following the same strategy applied to show Theorem 3.2 i) in \cite{BGJ2} it is possible to derive the next result. 
\begin{lem} \label{lemestfrac}
	Let $m \geq 1$. For every $\epsilon \in (0,1]$, let $O(\epsilon)$ and $Y(\epsilon)$ be given by \eqref{defOYeps}. Moreover, assume that there exist positive constants $K_0, K_1$ such that
	\begin{align}  \label{eq:mn}
		\forall i \geq 1, \forall H_1, \ldots, H_i \in N, \quad  \mathbb{E}_{\mathbb{Q}} \Big[ \max_{1 \leq \tilde{i} \leq i} \big\{ \ell^{(m)} _{\rho, \epsilon}(H_{\tilde{i}}) - K_0 \|  H_{\tilde{i}} \|_{Y(\epsilon)}^2 \big\} \Big]    \leq K_1.
	\end{align}
	Above, $K_0$ and $K_1$ are independent of $i$, $H_1, \ldots, H_i$ and $\epsilon$.  Furthermore, assume that %there exists a positive constant $K_2 >0$ such that 
	$ 
	\mathbb{E}_{\mathbb{Q}}  \big[ \| \rho^m - b^m \|_{Y} \big] < \infty. % \leq K_2.
	$
	Then% there exists a positive constant $\tilde K$ such that 
	$$
	\mathbb{E}_{\mathbb{Q}} \Big[  \int_0^T \iint_{ \mathbb{R}^2 } \frac{[ \rho^m(t,v) - \rho^m(t,u) ]^2}{|u-v|^{1+\gamma}} \, du \, dv \, dt  \Big] < \infty. % \leq \tilde K< \infty.
	$$
\end{lem}

\subsubsection{Useful bounds}
A particularly useful tool is the \textit{Dirichlet form} defined  by $\langle \sqrt{f},- n^{\gamma}\mcb  L \sqrt{f} \rangle_{\nu_b}$, where $f: \Omega \rightarrow \mathbb{R}$ is a density with respect to $\nu_b$ and for all functions $g,h: \Omega \rightarrow \mathbb{R}$, $\langle g,h \rangle_{\nu_b}$ denotes the scalar product in $L^{2} (\Omega, \nu_b )$. Now we define $\mcb D (\sqrt{f}, \nu_b )$ as
\begin{align} \label{defD}
	\mcb D (\sqrt{f}, \nu_b ) := \frac{1}{4}  \sum_{x,y } p( y - x )I_{x , y}  (\sqrt{f}, \nu_b ).
\end{align}  
where for every $x, y \in \mathbb{Z}$ 
\begin{align} \label{defIxy}
	I_{x,y}  (\sqrt{f}, \nu_b ) :=  \int_{\Omega}c^{(m)}_{x,y}(\eta)   \big[ \sqrt{f ( \eta^{x,y} ) } - \sqrt{f ( \eta ) } \big]^2 \, d \nu_b,
\end{align}
and  
$c^{(m)}_{x,y}(\eta)$ is given in \eqref{defrates}. From  \eqref{invexch} it is easy to check that 
\begin{align} \label{bound}
	\langle \mcb{L} \sqrt{f} , \sqrt{f} \rangle_{\nu_b} = -\frac{1}{2}   \mcb D (\sqrt{f}, \nu_b ),
\end{align}
for any density $f
$. 
Finally, we are ready to prove that $\rho$ satisfies the second condition of weak solutions in Definition \ref{eq:dif}. We present the proof only for the case $m \geq 2$, since for $m=1$ it is enough to apply Proposition 5.2 of \cite{CGJ2}.
\begin{prop} \label{estenergnlin}
	Assume \eqref{defcb} and $m \geq 2$. Then 
	\begin{align*}
		\mathbb{Q} \Big (  \rho^m - b^m  \in L^2 \big( 0, T ;  \mcb{H}^{\gamma/2}(\mathbb{R}) \big)  \Big) = 1.
	\end{align*}
\end{prop}
\begin{proof}
	From Proposition \ref{propestenergstat}  it is enough to prove that
	\begin{align*}
		\mathbb{Q} \Big( \int_0^T \iint_{ \mathbb{R}^2 } \frac{[\rho^m_t(v) - \rho^m_t(u) ]^2}{|u-v|^{1+\gamma}} du \, dv \, dt   < \infty \Big) = 1.
	\end{align*}
	We do so by applying Lemma \ref{lemestfrac}. From Corollary \ref{corestenergstat}, we are done if we can prove \eqref{eq:mn}. With this in mind, fix $\epsilon \in (0,1]$, $j \geq 1$ and $G_1, \ldots, G_j \in N$. Now for every $\varepsilon, K_0 >0$, define $\Psi_{K_0,\varepsilon}(\pi):=  \max_{1 \leq i \leq j} \big\{ - K_0 \|G_i \|_{Y(\epsilon)} + I_{\varepsilon}(\pi, G_i) \big\} \in \mathbb{R}$, for any $\pi \in \mcb D([0,T], \mcb M^{+})$, where 
	\begin{align*}
		I_{\varepsilon}(\pi, G_i):= & \int_0^T  \iint_{O(\epsilon)}  \frac{ \Big[ \prod_{j=0}^{m-1} \langle \pi_t, \overrightarrow{i_{\varepsilon}^{ v + j \varepsilon } }\rangle - \prod_{j=0}^{m-1} \langle \pi_t, \overrightarrow{i_{\varepsilon}^{ u + j \varepsilon } }\rangle \Big] G_i(t,u,v) }{|u-v|^{1+\gamma}}  du \, dv \,  dt. 
	\end{align*}
	Following \cite{CDG}, since $0 \leq \rho \leq 1$ and $G_1, \ldots, G_j$ are fixed, we have from \eqref{defdens} and Lebesgue's Differentiation Theorem that
	\begin{align*}
		\varlimsup_{\varepsilon \rightarrow 0^{+}} \mathbb{E}_{\mathbb{Q}} \big[ \max_{1 \leq i \leq j} \{ \ell_{\rho}^{(m,\epsilon)} (G_i) - I_{\varepsilon}(\pi, G_i) \} \big]=0.
	\end{align*}
	Therefore, in order to obtain \eqref{eq:mn}, it is enough to prove that there exist $K_0, K_1$ such that $\varlimsup_{\varepsilon \rightarrow 0^{+}} \mathbb{E}_{\mathbb{Q}} [ \Psi_{K_0,\varepsilon}(\pi) ] \leq K_1$. Next for every $K_0, \varepsilon >0$, every $n \geq 1$ and every $G \in N$, define the random application $\tilde{Z}^{K_0, \varepsilon}_{n,G}: \mcb D ( [0,T], \Omega ) \rightarrow \mathbb{R}$ by
	\begin{align*}
		\tilde{Z}^{K_0, \varepsilon}_{n,G}:=     \int_0^T \Big\{ n^{\gamma-1} \sum_{ (x, y) \in  I^n(\epsilon) } &\Big[  \prod_{j=0}^{m-1} \overrightarrow{\eta}_t^{\varepsilon n}(y+ j \varepsilon n) - \prod_{j=0}^{m-1} \overrightarrow{\eta}_t^{\varepsilon n}(x + j \varepsilon n) \Big] \times \\
		\times & G(t, \tfrac{x}{n}, \tfrac{y}{n})|y-x|^{- \gamma - 1} \Big\} \, dt - K_0 \|G \|_{Y(\epsilon)},
	\end{align*}
	where $I^n(\epsilon):= \{ (x,y) \in  \cap \mathbb{Z}^2: |x-y| \geq \epsilon n \}$. Since $\mathbb{Q}$ is a limit point, there exists a subsequence  $(\mathbb{Q}_{n'})_{n' \geq 1}$ of $(\mathbb{Q}_{n})_{n \geq 1}$ converging weakly to $\mathbb{Q}$. Combining this with the fact that $\mathbb{Q}_n$ is induced by $\mathbb{P}_{\mu_n}$ for every $n \geq 1$ and that $\Psi_{K_0,\varepsilon}$ is bounded and lower semi-continuous with respect to the Skorohod topology of $\mcb{D} ([0,T], \mcb{M}^{+} )$ for every $K_0, \varepsilon > 0$, we get
	\begin{align*}
		&\varlimsup_{\varepsilon \rightarrow 0^{+}} \mathbb{E}_{\mathbb{Q}} [ \Psi_{K_0,\varepsilon}(\pi) ] = \varlimsup_{\varepsilon \rightarrow 0^{+}} \lim_{n' \rightarrow \infty} \mathbb{E}_{\mathbb{Q}_{n'}} [  \Psi_{K_0,\varepsilon}(\pi)   ] \\
		=& \varlimsup_{\varepsilon \rightarrow 0^{+}} \lim_{n' \rightarrow \infty} \mathbb{E}_{\mu_{n'}} \big[  \max_{1 \leq i \leq j} \{ \tilde{Z}^{K_0}_{n',G_i} \}  \big] \leq  \varlimsup_{\varepsilon \rightarrow 0^{+}} \varlimsup_{n \rightarrow \infty} \mathbb{E}_{\mu_{n}} \big[  \max_{1 \leq i \leq j} \{ \tilde{Z}^{K_0}_{n,G_i} \}  \big].
	\end{align*}
	This means that we are done if we can prove that
	\begin{equation} \label{claim2}
		\exists K_0, K_1 >0: \quad \varlimsup_{\varepsilon \rightarrow 0^{+}} \varlimsup_{n \rightarrow \infty} \mathbb{E}_{\mu_{n}} \big[  \max_{1 \leq i \leq j} \{ \tilde{Z}^{K_0}_{n,G_i} \}  \big] \leq K_1.
	\end{equation}
	We do so by applying the strategy in the proof of Proposition 5.3 in \cite{CDG}. For every $G \in N$, we define $(\Phi_n^G)_{n \geq 1}: [0,T] \times \mathbb{R} \rightarrow \mathbb{R}$ by
	\begin{align*}
		\forall t \in [0,T], \forall x \in \mathbb{Z}, \quad \Phi_n^G (t, \tfrac{x}{n} ) := n^{\gamma} \sum_{y: |x-y| \geq \epsilon n} G(t, \tfrac{x}{n}, \tfrac{y}{n}) |y-x|^{- \gamma - 1}. 
	\end{align*}
	Following the arguments in \cite{CDG}, we have that $(\Phi_n)^G_{n \geq 1}$ satisfies \eqref{boundrep}.  From Lemma \ref{globrep}, for every $G \in N$, it holds
	\begin{align*}
		\varlimsup_{\varepsilon \rightarrow 0^{+}} \varlimsup_{n \rightarrow \infty} \mathbb{E}_{\mu_n} \Big[   \int_0^T n^{\gamma-1} \sum_{x,y:|x-y| \geq \epsilon n}  A_{x,y}(\varepsilon n, \, \eta_t^n) G(t, \tfrac{x}{n}, \tfrac{y}{n}) |y-x|^{- \gamma - 1} \, dt \Big] =0,
	\end{align*}
	where for every $\eta \in \Omega$, any $x,y \in \mathbb{Z}$ and every $\ell \geq 1$, $ A_{x,y}(\ell, \eta)$ is given by
	\begin{align*}
		A_{x,y}(\ell, \eta) :=    \prod_{j=0}^{m-1} \overrightarrow{\eta}^{\ell}(x + j \ell)  - \prod_{j=0}^{m-1}  \eta(x+j)  +  \prod_{j=0}^{m-1} \overrightarrow{\eta}^{\ell}(y + j \ell)  - \prod_{j=0}^{m-1}  \eta(y+j).
	\end{align*}
	for any $G \in N$. Then in order to get \eqref{claim2}, it is enough to find $K_0, K_1 > 0$ such that
	\begin{align} \label{claim3}
		\varlimsup_{n \rightarrow \infty} \mathbb{E}_{\mu_n} \big[ \max_{1 \leq i \leq j} \{  Z^{K_0}_{n,G_i}   \}\big] \leq K_1,
	\end{align}
	where for any $K_0$, $n \geq 1$ and $G \in N$, $Z^{K_0}_{n,G}: \mcb D ( [0,T], \Omega ) \rightarrow \mathbb{R}$ is given by
	\begin{align*}
		Z^{K_0}_{n,G}:=     \int_0^T \Big\{ & n^{\gamma-1} \sum_{ (x, y) \in  I^n(\epsilon) }  \tilde{A}_{x,y}( \eta_t^n) G(t, \tfrac{x}{n}, \tfrac{y}{n})|y-x|^{- \gamma - 1}   -   \frac{K_0 \| G \|^2_{Y(\epsilon)}}{T}  \Big\} \, dt.
	\end{align*}
	Above, for any $\eta \in \Omega$ and every $x,y \in \mathbb{Z}$, $ \tilde{A}_{x,y}(\eta)$ is given by
	\begin{align*}
		\tilde{A}_{x,y}( \eta) :=  \prod_{j=0}^{m-1}  \eta(y+j) - \prod_{j=0}^{m-1}  \eta(x+j)   .
	\end{align*}
	Combining entropy and Jensen's inequalities with the inequality
	\begin{align*} % \label{inmaxlimsup}
		\varlimsup_{n \rightarrow \infty} \big( \max \{a_n, b_n\} \big)_{n \geq 1} \leq \max \big\{ \varlimsup_{n \rightarrow \infty} (a_n)_{n \geq 1}, \varlimsup_{n \rightarrow \infty} (b_n)_{n \geq 1} \big\}, 
	\end{align*}
	we get the following estimate.
	\begin{align*}
		\varlimsup_{n \rightarrow \infty} \mathbb{E}_{\mu_n} \big[ \max_{1 \leq i \leq j} \{  Z^{K_0}_{n,G_i}   \}\big]] \leq &  C_b + \max_{1 \leq i \leq j}  \varlimsup_{n \rightarrow \infty} \frac{1}{n} \log \Big(  \mathbb{E}_{\nu_b} \big[ \exp  (n Z^{K_0}_{n,G_i})  \big]  \Big) \Big) .
	\end{align*}
	Therefore, in order to show that \eqref{claim3} holds for $K_1=C_b$ and end the proof, it is enough to prove the following claim:
	\begin{equation}  \label{cladynnonlin}
		\exists K_0 >0:  \quad \sup_{G \in N} \varlimsup_{n \rightarrow \infty} \frac{1}{n} \log \Big(  \mathbb{E}_{\nu_b} \big[ \exp  (n Z^{K_0}_{n,G})  \big]  \Big) \Big) \leq 0.
	\end{equation}
	Keeping this in mind, we apply Feynman-Kac's formula (see Lemma A.1 in \cite{baldasso}) and bound the last $\varlimsup$ from above by 
	\begin{equation} \label{est2}
		\begin{split}
			\varlimsup_{n \rightarrow \infty} \int_0^T \sup_{f}   \Big\{ &  \sum_{x,y:|x-y| \geq \epsilon n} \frac{n^{\gamma-1}}{|y-x|^{ \gamma + 1}}  \int_{\Omega}\tilde{A}_{x,y}( \eta) G_t(\tfrac{x}{n}, \tfrac{y}{n})  f(\eta)\,d\nu_{b}  \\
			+ &n^{\gamma-1}\langle \mcb L \sqrt{f},\sqrt{f} \rangle_{\nu_{b}} - \frac{K_0 \| G \|^2_{Y(\epsilon)}}{T} \Big\} \, dt,  
		\end{split}
	\end{equation}
	where last supremum is carried over all the densities $f$ with respect to $\nu_{b}$. Recall \eqref{defrates}; in order to apply \eqref{bound}, the sum inside the supremum in \eqref{est2} will be given in terms of
	\begin{align*}
		\sum_{x,y:|x-y| \geq \epsilon n} \frac{n^{\gamma-1}}{|y-x|^{ \gamma + 1}} G_t (  \tfrac{x}{n},\tfrac{y}{n})  \int_{\Omega} [\eta (y) - \eta (x)] c^{(m)}_{x,y}( \eta) f(\eta)\,d\nu_{b}.
	\end{align*}
	For $n \geq 2 \epsilon^{-1}$, $\varepsilon n \geq 2$ and $c^{(m)}_{x,y}( \eta)=c^{(m,dif)}_{x,y}(\eta) + c^{(m,dif)}_{y,x}(\eta)$ when $|x-y| \geq \epsilon n$. Next define $\hat{A}_{m}( \eta, x, y  ):=C_m(\eta,y,x) - C_m(\eta,x,y)$ for every $\eta \in \Omega$ and every $x,y \in \mathbb{Z}$, where $C_m(\cdot,\cdot,\cdot)$ is given in \eqref{defCm}. Then from \eqref{algman}, the sum inside the supremum in \eqref{est2} is equal to
	\begin{align*}
		&   \sum_{x,y:|x-y| \geq \epsilon n} \frac{n^{\gamma-1}}{2|y-x|^{ \gamma + 1}} G_t (  \tfrac{x}{n},\tfrac{y}{n} )  \int_{\Omega} [\eta (y) - \eta (x)] c^{(m)}_{x,y}( \eta) \,  d\nu_{b}   \\
		+&     \sum_{x,y:|x-y| \geq \epsilon n} \frac{n^{\gamma-1}}{2|y-x|^{ \gamma + 1}}  [G_t( \tfrac{x}{n}, \tfrac{y}{n}) - G_t( \tfrac{x+m-1}{n}, \tfrac{y+m-1}{n}) ]   \int_{\Omega} \tilde{A}_{x,y}( \eta)  f(\eta)\,d\nu_{b}  \\
		+&  \sum_{x,y:|x-y| \geq \epsilon n} \frac{n^{\gamma-1}}{2|y-x|^{ \gamma + 1}} [G_t( \tfrac{x-1}{n}, \tfrac{y-1}{n}) - G_t( \tfrac{x}{n}, \tfrac{y}{n}) ]   \int_{\Omega} \hat{A}_{m}( \eta, x, y  )  f(\eta) \, d\nu_{b}.
	\end{align*}
	In particular, since $f$ is a density with respect to $\nu_{b}$ and there exists $b_G >0$ such that $G(t,u,v)=0$ when $|u| \geq b_G$ or $|v| \geq b_G|$, the sum of the terms in the second and third lines of last display is bounded from above by 
	\begin{align*}
		m C(G,m,n)  n^{\gamma-1} \sum_{x,y:|x-y| \geq \epsilon n, |x| + |y| \leq 4 b_G n } |y-x|^{- \gamma - 1} \lesssim m b_G \epsilon^{-\gamma} C(G,m,n),
	\end{align*} 
	where $C(G,m,n)$ is given by
	\begin{align*}
		C(G,m,n):= \sup_{t \in [0,T], \{(u_2-u_1)^2+ (v_2-v_1)^2 < 2m^2 n^{-2} \} } |G_t(u_2,v_2) - G_t(u_1,v_1)|. 
	\end{align*}
	Since $G$ is uniformly continuous and $\lim_{n \rightarrow \infty} C(G,m,n) =0$. Therefore we conclude that the display in \eqref{est2} is equal to
	\begin{equation} \label{est3}
		\begin{split}
			\varlimsup_{n \rightarrow \infty} \int_0^T \sup_{f}   \Big\{ &    \sum_{(x, y) \in I^n(\epsilon)}\frac{n^{\gamma-1}}{2|y-x|^{ \gamma + 1}}  G_t (  \tfrac{x}{n},\tfrac{y}{n} )   \int_{\Omega} [\eta (y) - \eta (x)] c^{(m)}_{x,y}( \eta) \,  d\nu_{b}   \\
			+ &n^{\gamma-1}\langle \mcb L \sqrt{f},\sqrt{f} \rangle_{\nu_{b}} - \frac{K_0 \| G \|^2_{Y(\epsilon)}}{T} \Big\} \, dt.
		\end{split}
	\end{equation}
	Next we bound the sum inside the supremum in \eqref{est3} from above by
	\begin{align*} % \label{sumest}
		\sum_{ (x, y) \in I^n(\epsilon) }  \frac{n^{\gamma-1}}{2|y-x|^{ \gamma + 1}}   \big|G_t ( \tfrac{x}{n}, \tfrac{y}{n} )\big|  \Big|  \int_{\Omega} [ \eta(y) - \eta(x) ] c^{(m)}_{x,y}( \eta) f(\eta) \,  d \nu_b \Big|.
	\end{align*}
	From \eqref{invexch},  the fact that $c^{(m)}_{x,y}( \eta)=c^{(m)}_{y,x}( \eta)$, and Young's inequality and \eqref{defIxy}, we get
	\begin{align*}
		&\Big|  \int_{\Omega} [ \eta(y) - \eta(x) ] c^{(m)}_{x,y}( \eta)  f(\eta)  \, d \nu_b \Big| \\
		% =& \Big|  \frac{1}{2} \int_{\Omega} [ \eta(y) - \eta(x) ]  c^{(m)}_{x,y}( \eta) [f(\eta) - f(\eta^{x,y})]   d \nu_b \Big| \\
		\leq &  \frac{1}{2} \int_{\Omega} c^{(m)}_{x,y}( \eta) |\sqrt{f(\eta)}-\sqrt{f(\eta^{x,y})} | [ \sqrt{f(\eta)}+\sqrt{f(\eta^{x,y})} ] \, d \nu_b \\
		\leq & \frac{m}{ E_{x,y}} \int_{\Omega}  [f(\eta) + f(\eta^{x,y})]   \, d \nu_b  + \frac{E_{x,y}}{4} I_{x,y}( \sqrt{f}, \nu_b ),
	\end{align*}
	for every $x,y \in \mathbb{Z}$ and every $E_{x,y} >0$. Above, we used the fact (which is due to \eqref{defrates} and Remark \ref{ratpor}) that $c_{x,y}^{(m)} \leq 2 m$. From \eqref{invexch} and the fact that $f$ is a density with respect to $\nu_b$, last display can be rewritten as
	\begin{align*}
		\frac{2m}{ E_{x,y}} + \frac{E_{x,y}}{4} I_{x,y}( \sqrt{f}, \nu_b ),
	\end{align*}
	therefore the sum inside the supremum in \eqref{est3} is bounded from above by
	\begin{equation} \label{dynnonlinbnd2}
		\sum_{ (x, y) \in I^n(\epsilon) } \frac{n^{\gamma-1}}{2|y-x|^{ \gamma + 1}}   \big|G_t ( \tfrac{x}{n}, \tfrac{y}{n} )\big|  \Big( \frac{E_{x,y}}{4} I_{x,y}( \sqrt{f}, \nu_b ) + \frac{2m}{ E_{x,y}} \Big).
	\end{equation}
	From \eqref{bound}, we rewrite the second term inside the supremum in \eqref{est3} as
	\begin{align*}
		- \frac{n^{\gamma-1}}{2} \mcb  D (\sqrt{f}, \nu_b) \leq - \frac{n^{\gamma-1}}{8} \sum_{ (x, y) \in I^n(\epsilon) }  c_{\gamma} |y-x|^{-1-\gamma}  I_{x,y}( \sqrt{f}, \nu_b ) .
	\end{align*}
	Combining last display with \eqref{dynnonlinbnd2}, we conclude that the expression inside the supremum in \eqref{est3} is bounded from above by
	\begin{align*}
		& \frac{n^{\gamma-1}}{8}  \sum_{ (x, y) \in I^n(\epsilon) } |y-x|^{-1-\gamma}  I_{x,y}( \sqrt{f}, \nu_b ) \Big[ E_{x,y}|G_t ( \tfrac{x}{n}, \tfrac{y}{n} )| - c_{\gamma} \Big] \\
		& - \frac{K_0 \| G \|^2_{  Y(\epsilon)}}{T}  + n^{\gamma-1}   \sum_{ (x, y) \in I^n(\epsilon) }   |G_t ( \tfrac{x}{n}, \tfrac{y}{n} )| |y-x|^{-1-\gamma}  \frac{m}{ E_{x,y}}.
	\end{align*}
	Choosing $E_{x,y} =  c_{\gamma} ( |G_t ( \tfrac{x}{n}, \tfrac{y}{n} )|)^{-1}$, last display can be rewritten as
	\begin{align*}
		- \frac{K_0 \| G \|^2_{  Y(\epsilon)}}{T} + \frac{m}{c_{\gamma}} n^{\gamma-1}   \sum_{ (x, y) \in I^n(\epsilon) }   [G_t ( \tfrac{x}{n}, \tfrac{y}{n} )]^2 |y-x|^{-1-\gamma},
	\end{align*}
	therefore the display in \eqref{est3} is bounded from above by
	\begin{align*}
		&\varlimsup_{n \rightarrow \infty} \int_0^T \Big\{ - \frac{K_0 \| G \|^2_{  Y(\epsilon)}}{T} + \frac{m}{c_{\gamma}} \frac{1}{n}^2   \sum_{ (x, y) \in I^n(\epsilon) }   [G_t ( \tfrac{x}{n}, \tfrac{y}{n} )]^2 \Big| \frac{y-x}{n} \Big|^{-1-\gamma} \Big\} \, dt \\
		=& - K_0 \| G \|^2_{  Y(\epsilon)} + \frac{m}{c_{\gamma}} \int_0^T \iint_{O(\epsilon)} \frac{[G_t ( \tfrac{x}{n}, \tfrac{y}{n} )]^2}{|u-v|^{1+\gamma}} du \; dv \; dt \\
		=& \Big(  \frac{m}{c_{\gamma}} -K_0 \Big) \| G \|^2_{  Y(\epsilon)}.
	\end{align*}
	Finally, by choosing $K_0 = m (c_{\gamma})^{-1}$, we get \eqref{cladynnonlin} and the proof ends.
\end{proof}

\subsection{Replacement lemmas} \label{sec:replem}

In this subsection, we present the proof of Lemma \ref{globrep} (we refer the interested reader to Section 6 of \cite{CDG} for more details). Let $\varepsilon >0$ and $n \geq 1$ such that $\varepsilon n^{\gamma/2} \geq 1$. In the first step, similarly to Lemma 5.3 in \cite{bonorino} and Lemma 6.2 in \cite{CDG}, only jumps of size $1$ are used to replace products of $\eta_s^n$'s by products of empirical averages in boxes of size $\ell=\varepsilon n^{\gamma/2}$. The later products are then  {replaced} by products of empirical averages in boxes of size $L=\varepsilon n$ in the second step, by applying long jumps. It is crucial that our dynamics always allows exchanges of particles between neighbor sites. For $m=2$, this is assured by Remark \ref{remsep} (this remark was applied often in Section 6 of \cite{CDG}). For $m \geq 3$, this is assured by \eqref{defrates}.

Next we illustrate the strategy for the first step. Given $\eta \in \Omega$ and $x \in \mathbb{Z}$, observe that for $m=2$, it holds
\begin{align*}
	&\prod_{i=0}^{m-1} \eta(x+i ) - \prod_{i=0}^{m-1} \overrightarrow{\eta}^{\ell} (x+  i \ell) = \eta(x ) \eta(x+1 ) - \overrightarrow{\eta}^{\ell} (x)  \overrightarrow{\eta}^{\ell} (x+ \ell) \\
	=& \eta(x ) [ \eta(x+1 ) - \overrightarrow{\eta}^{\ell} (x+ \ell)] + \overrightarrow{\eta}^{\ell} (x+ \ell) [\eta(x ) -  \overrightarrow{\eta}^{\ell} (x) ] .
\end{align*}
%For $m=3$, we have
%\begin{align*}
%&\prod_{i=0}^{m-1} \eta(x+i ) - \prod_{i=0}^{m-1} \overrightarrow{\eta}^{\ell} (x+  i \ell) \\
%=& \eta(x ) \eta(x+1 ) \eta(x+2 ) - \overrightarrow{\eta}^{\ell} \big(x)  \overrightarrow{\eta}^{\ell} (x+ \ell) (x)  \overrightarrow{\eta}^{\ell} (x+ 2 \ell) \\
%=& \eta(x ) \eta(x+1 )  [ \eta(x+2 ) - \overrightarrow{\eta}^{\ell} (x+ 2 \ell)] + \eta(x ) \overrightarrow{\eta}^{\ell} (x+ 2 \ell) [\eta(x +1 ) -  \overrightarrow{\eta}^{\ell} (x+ \ell) ] \\
%+&   \overrightarrow{\eta}^{\ell} (x+  \ell)  \overrightarrow{\eta}^{\ell} (x+ 2 \ell) [\eta(x ) -  \overrightarrow{\eta}^{\ell} (x) ].
%\end{align*}
More generally, for every $m \geq 2$, it holds
\begin{align*}
	&\prod_{i=0}^{m-1} \eta(x+i ) - \prod_{i=0}^{m-1} \overrightarrow{\eta}^{\ell} (x+  i \ell)  = \sum_{j=1}^{m} B^{\ell}_{x,j}(\eta) \big[\eta(x+m-j)- \overrightarrow{\eta}^{\ell}\big(x+(m-j) \ell  \big) \big],
\end{align*}
where for every $x \in \mathbb{Z}$, $\eta \in \Omega$ and $j \in \{1, \ldots, m\}$, $B^{\ell}_{x,j}(\eta)$ is defined by 
\begin{align} \label{defbxjeta}
	B^{\ell}_{x,j}(\eta):= \prod_{i=0}^{m-j-1} \eta(x+i)  \prod_{i=m-j+1}^{k-1} \overrightarrow{\eta}^{\ell}(x +i \ell ).
\end{align}
This motivates us to state next result. %Its proof is analogous to the proof of Lemma 5.3 in \cite{bonorino}, therefore we omit most of the steps.
\begin{lem} \label{lemrep1}
	Let $t \in [0,T]$. Assume $(\Phi_{n})_{n \geq 1}: [0,T] \times \mathbb{R} \rightarrow \mathbb{R}$ satisfies \eqref{boundrep} and denote $\ell = \varepsilon n^{\gamma/2}$. Then for every $j \in \{1, \ldots, m\}$, it holds
	\begin{equation*} % \label{rl1rl}
		\varlimsup_{\varepsilon \rightarrow 0^{+}}\varlimsup_{n \rightarrow \infty}\mathbb{E}_{\mu_n} \Big[ \Big| \int_{0}^{t} \sum_{x} \Phi_{n} (s, \tfrac{x}{n} ) \frac{B^{\ell}_{x,j}(\eta_s^n)}{n}  \big[\eta_{s}^n(x+m-j)- \overrightarrow{\eta}_{s}^{\ell}\big(x+(m-j) \ell  \big) \big] \, ds \, \Big| \Big] =0.
	\end{equation*}
	In particular, we have
	\begin{equation*}  
		\varlimsup_{\varepsilon \rightarrow 0^{+}}\varlimsup_{n \rightarrow \infty}\mathbb{E}_{\mu_n} \Big[ \Big| \int_{0}^{t}\frac{1}{n}\sum_{x} \Phi_{n} (s, \tfrac{x}{n} )\Big( \prod_{i=0}^{m-1} \eta^n_{s}(x+i) - \prod_{i=0}^{m-1} \overrightarrow{\eta}_{s}^{\ell}(x+ i \ell ) \Big) \,  ds \, \Big| \Big] =0.
	\end{equation*}
\end{lem}
\begin{proof}
	Combining entropy and Jensen's inequalities with Feynman-Kac's formula, we are done if we can prove that, for every $j \in \{1, \ldots, m\}$, 
	\begin{equation}\label{rl4}
		\begin{split}
			\sup_{f} \Big\{  \Big| & \frac{1}{n} \sum_{x} \Phi_{n} (s, \tfrac{x}{n} ) \int_{\Omega} B^{\ell}_{x,j}(\eta) \big[\eta(x+m-j)- \overrightarrow{\eta}^{\ell}\big(x+(m-j) \ell  \big) \big]\, d \nu_b \Big| \\
			+& n^{\gamma-1} \varepsilon \langle \mcb L\sqrt{f},\sqrt{f} \rangle_{ \nu_b } \Big\} ,
		\end{split}
	\end{equation}
	goes to zero, when we take first the $\limsup$ for $n \rightarrow \infty$ and afterwards the $\limsup$ for $\varepsilon \rightarrow 0^{+}$. In \eqref{rl4}, the supremum is carried over all densities $f$ with respect to $\nu_b$. In order to treat the first term inside this supremum, we observe that% for every $j \in \{1, \ldots, m\}$, it holds
	%\begin{align*}
	%  \eta(x+m-j)- \overrightarrow{\eta}^{\ell}\big(x+(m-j) \ell  \big)  =& \frac{1}{\ell} \sum_{y=1}^{\ell}[ \eta(x+m-j) - \eta \big( x+(m-j) \ell + y \big)\\ 
	%  = & \frac{1}{\ell} \sum_{y=1}^{\ell} \sum_{z=x+m-j}^{x+(m-j) \ell + y -1} [\eta(z) - \eta(z+1) ].
	%\end{align*}
	\begin{align*}
		\eta(x+m-j)- \overrightarrow{\eta}^{\ell}\big(x+(m-j) \ell  \big)   =  \frac{1}{\ell} \sum_{y=1}^{\ell} \sum_{z=x+m-j}^{x+(m-j) \ell + y -1} [\eta(z) - \eta(z+1) ].
	\end{align*}
	Therefore, the first term inside the supremum in \eqref{rl4} is bounded from above by
	\begin{equation}\label{rl7}
		\Big| \sum_{x} \Phi_{n} (s, \tfrac{x}{n} ) \sum_{y=1}^{\ell} \sum_{z=x+m-j}^{x+(m-j) \ell + y -1} \int_{\Omega} \frac{B^{\ell}_{x,j}(\eta)}{n\ell} [\eta(z)-\eta(z+1)]  f(\eta) \, d \nu_b \, \Big|.
	\end{equation}
	From \eqref{defbxjeta}, we have that $B_{x,j}(\eta)=B_{x,j}(\eta^{z,z+1})$ when $x+m-j \leq z \leq x+(m-j) \ell + y -1$ and $1 \leq y \leq \ell$, for every $x \in \mathbb{Z}$, $j \in \{1, \ldots, m\}$ and $\eta \in \Omega$. Then, by writing $f(\eta)=\frac{1}{2}f(\eta)+\frac{1}{2}f(\eta)$ and performing the change of variables $\eta \rightarrow \eta^{z,z+1}$, the display in \eqref{rl7} is equal to
	\begin{equation*}
		\Big| \sum_{x} \Phi_{n} (s, \tfrac{x}{n} ) \sum_{y=1}^{\ell} \sum_{z=x+m-j}^{x+(m-j) \ell + y -1}\int_{\Omega}  \frac{B^{\ell}_{x,j}(\eta)}{2n\ell} [\eta(z)-\eta(z+1)] [f(\eta)-f(\eta^{z,z+1}) ] \,d \nu_b  \Big|.
	\end{equation*} 
	By Young's inequality and \eqref{boundrep}, last display is bounded from above by
	\begin{align}
		%&\frac{1}{4An\ell} \sum_{x} | \Phi_{n} (s, \tfrac{x}{n} ) |\sum_{y=1}^{\ell} \sum_{z=x+m-j}^{x+(m-j) \ell + y -1} \int_{\Omega} \big[\sqrt{f(\eta)}+\sqrt{f(\eta^{z,z+1})} \big]^2 \, d\nu_{b}   \nonumber \\
		%+&\frac{A  }{4n\ell } \sum_{x} | \Phi_{n} (s, \tfrac{x}{n} ) | \sum_{y=1}^{\ell} \sum_{z=x+m-j}^{x+(m-j) \ell + y -1}  \int_{\Omega}  [\sqrt{f(\eta)}-\sqrt{f(\eta^{z,z+1})} ]^2 d\nu_{b} \nonumber \\
		& \frac{m \ell}{A } M_1 + \frac{A  M_2}{4n\ell p(1)} \sum_{x} \sum_{y=1}^{\ell} \sum_{z=x+m-j}^{x+(m-j) \ell + y -1}   p(1) I_{z,z+1} ( \sqrt{f}, \nu_{b} )
		\label{rlnl1b}.
	\end{align}
	for any $A>0$. Above, we used the facts that $[ B_{x,j}(\eta) ]^2  [\eta(z)-\eta(z+1)]^2 \leq 1$ (due to \eqref{defbxjeta} and $\eta(\cdot) \leq 1$) and $f$ is a density with respect to $\nu_b$. In order to apply \eqref{bound}, it is crucial to count the maximum number of times that a fixed bond $\{z_0,z_0+1\}$ appears in the triple sum in \eqref{rlnl1b}.
	
	Observe that for $x \leq z_0 - m \ell$ or $x> z_0$ this bond does not appear. And for every $x \in \{z_0 - m \ell +1, \ldots, z_0 \}$, $\{z_0,z_0+1\}$ is counted at most $\ell$ times (since $y$ ranges over $\{1, \ell\}$). Hence every bond $\{z_0,z_0+1\}$ is counted at most $4m \ell^2$ times in the triple sum in \eqref{rlnl1b}. Therefore, from \eqref{defD} and \eqref{bound}, the expression inside the supremum in \eqref{rl4} is bounded from above by
	\begin{align*}
		\frac{m \ell}{A } M_1 + \Big[ A\frac{M_2 m \ell}{4n p(1)} - \frac{n^{\gamma}}{n} \varepsilon \Big] \mcb D (\sqrt{f}, \nu_b ),
	\end{align*}
	for any $A>0$. Since $\ell=\varepsilon n^{\gamma/2}$, by choosing $A= 4 p(1) n^{\gamma} \varepsilon (M_2 m \ell)^{-1}$, last display is bounded from above by
	\begin{align*}
		\frac{M_1 M_2 m^2}{4  p(1)} \frac{\ell^2}{n^{\gamma} \varepsilon} \lesssim \frac{\ell^2}{n^{\gamma} \varepsilon} = \frac{(\varepsilon n^{\gamma/2})^2}{n^{\gamma} \varepsilon} = \varepsilon,
	\end{align*}
	which vanishes as $\varepsilon \rightarrow 0^+$, and the proof ends. 
\end{proof}
Next we state an auxiliary lemma which will be useful in the proof of Lemma \ref{lemrep2}. We omit its proof, since the strategy is analogous to the one used to show Lemma 6.3 in \cite{CDG}.   
\begin{lem} \textbf{(Moving Particle Lemma)} \label{mpl} 
	Fix $r \neq 0 \in \mathbb{Z}$ and $f$ a density with respect to $\nu_{b}$ on $\Omega$. For every $x \in \mathbb{Z}$, let  $\Omega_x:=\{  \eta \in \Omega : \eta(x+j) =1, j =1,2 , \ldots, k \}$. Then
	\begin{align*}
		\sum_{x} \int_{\Omega_x} \big[\sqrt{f \left( \eta^{x,x+r} \right) } - \sqrt{f \left( \eta \right) } \big]^2 \, d \nu_{b} \lesssim  |r|^{\gamma} \mcb D (\sqrt{f}, \nu_{b} ) .
	\end{align*}
\end{lem}
Now, we illustrate the strategy for the second step. For every $m \geq 2$, it holds
\begin{align*}
	&\prod_{i=0}^{m-1} \overrightarrow{\eta}^{\ell} (x+  i \ell) - \prod_{i=0}^{m-1} \overrightarrow{\eta}^{\varepsilon n} (x+  i \varepsilon n)  \\
	=& \sum_{j=1}^{m} \tilde{B}^{\ell}_{x,j}(\eta) \big[\overrightarrow{\eta}^{\ell}\big(x+(m-j) \ell  \big)- \overrightarrow{\eta}^{\varepsilon n}\big(x+(m-j) \varepsilon n  \big) \big],
\end{align*}
where $\ell \leq \varepsilon n$ and for every $x \in \mathbb{Z}$, $\eta \in \Omega$ and $j \in \{1, \ldots, m\}$, $\tilde{B}^{\ell}_{x,j}(\eta)$ is defined by 
\begin{align} \label{defbtilxjeta}
	\tilde{B}^{\ell}_{x,j}(\eta):= \prod_{i=0}^{m-j-1} \overrightarrow{\eta}^{\ell}(x +i \ell )  \prod_{i=m-j+1}^{k-1} \overrightarrow{\eta}^{\varepsilon n}(x +i \varepsilon n ).
\end{align}
This motivates us to state next result. The proof is analogous to the proof of Lemma 6.5 in \cite{CDG}, therefore we omit some steps.
\begin{lem} \label{lemrep2}
	Assume $(\Phi_{n})_{n \geq 1}: [0,T] \times \mathbb{R} \rightarrow \mathbb{R}$ satisfies \eqref{boundrep} and denote $\ell = \varepsilon n^{\gamma/2}$. Then for every $j \in \{1, \ldots, m\}$, it holds
	\begin{equation*}%  \label{rl1rl}
		\begin{split}
			\varlimsup_{\varepsilon \rightarrow 0^{+}}\varlimsup_{n \rightarrow \infty}\mathbb{E}_{\mu_n} \Big[ \Big| \int_{0}^{t}\frac{1}{n}\sum_{x} & \Phi_{n} (s, \tfrac{x}{n} ) \tilde{B}^{\ell}_{x,j}(\eta_s^n) \\
			& \times \big[ \overrightarrow{\eta}_{s}^{\ell}\big(x+(m-j) \ell  \big) - \overrightarrow{\eta}_{s}^{\varepsilon n}\big(x+(m-j) \varepsilon n  \big) \big] \, ds \, \Big| \Big] =0.
		\end{split}
	\end{equation*}
	In particular, we have
	\begin{equation*}  
		\varlimsup_{\varepsilon \rightarrow 0^{+}}\varlimsup_{n \rightarrow \infty}\mathbb{E}_{\mu_n} \Big[ \Big| \int_{0}^{t}\frac{1}{n}\sum_{x} \Phi_{n} (s, \tfrac{x}{n} )\Big[ \prod_{i=0}^{m-1} \overrightarrow{\eta}_{s}^{\ell}(x+i \ell ) - \prod_{i=0}^{m-1} \overrightarrow{\eta}_{s}^{\varepsilon n}(x+i  \varepsilon n ) \Big]  \, ds\, \Big| \Big] =0.
	\end{equation*}
\end{lem}
\begin{proof}
	By applying the initial steps of the proof of Lemma \ref{lemrep1}, we are done if we can prove that, for every $j \in \{1, \ldots, m\}$, 
	\begin{equation}\label{expect1b}
		\begin{split}
			\sup_{f} \Big \{& \Big|  \int_{\Omega}\frac{1}{n}\sum_{x } \Phi_n(s, \tfrac{x}{n} )\tilde{B}^{\ell}_{x,j}(\eta) \big[ \overrightarrow{\eta}^{\ell}\big(x+(m-j) \ell  \big) - \overrightarrow{\eta}^{\varepsilon n}\big(x+(m-j) \varepsilon n  \big) \big] \, d \nu_b \,  \Big| \\
			+& n^{\gamma-1} \varepsilon^{\gamma/2} \langle \mcb L \sqrt{f}, \sqrt{f} \rangle_{\nu_{b}} \Big\} ,
		\end{split}
	\end{equation}
	goes to zero, when we take first the $\limsup$ for $n \rightarrow \infty$ and afterwards the $\limsup$ for $\varepsilon \rightarrow 0^{+}$. In \eqref{expect1b}, the supremum is carried over all densities $f$ with respect to $\nu_b$. In order to treat the first term inside this supremum, we write $\varepsilon n = k \ell$ and observe that
	\begin{align*}
		&\overrightarrow{\eta}^{\ell} \big(x+(m-j) \ell  \big) - \overrightarrow{\eta}^{\varepsilon n}\big(x+(m-j) \varepsilon n  \big) \\
		=& \frac{1}{k \ell} \sum_{i=0}^{k-1} \sum_{y=1}^{\ell} \big[ \eta \big(x+(m-j) \ell +y \big) - \eta \big(x+(m-j) \ell +y + \ell [i + (m-j) (k-1)] \big) \big] \\
		=& \frac{1}{k \ell} \sum_{i=0}^{k-1} \sum_{y=1}^{\ell} [ \eta (z_{x,y}) - \eta \big(z_{x,y} + \ell_i ) ],
	\end{align*}
	with $z_{x,y}:=x+(m-j) \ell+y$ and $\ell_{i}:=\ell [i + (m-j) (k-1)]$. Thus, the first term inside the supremum in \eqref{expect1b} is equal to
	\begin{align} \label{rl7b}
		\Big|  \sum_{x } \Phi_n(s, \tfrac{x}{n} ) \sum_{i=0}^{k-1} \sum_{y=1}^{\ell}   \int_{\Omega} \frac{ \tilde{B}^{\ell}_{x,j}(\eta) }{k \ell n}  [ \eta (z_{x,y}) - \eta (z_{x,y} + \ell_i ) ] f(\eta) \, d\nu_{b} \, \Big|.
	\end{align}
	From \eqref{defbtilxjeta}, we have that $\tilde{B}^{\ell}_{x,j}(\eta)=\tilde{B}^{\ell}_{x,j}(\eta^{z_{x,y},z_{x,y}+\ell_i})$ when $0 \leq i \leq k -1$ and $1 \leq y \leq \ell$, for every $x \in \mathbb{Z}$, $j \in \{1, \ldots, m\}$ and $\eta \in \Omega$. Then, by writing $f(\eta)=\frac{1}{2}f(\eta)+\frac{1}{2}f(\eta)$ and performing the change of variables $\eta \rightarrow \eta^{z_{x,y},z_{x,y}+\ell_i}$, the display in \eqref{rl7b} is equal to
	\begin{equation}\label{expr3b}
		\Big|  \sum_{x } \Phi_n(s, \tfrac{x}{n} ) \sum_{i=0}^{k-1} \sum_{y=1}^{\ell}   \int_{\Omega} \frac{\tilde{B}^{\ell}_{x,j}(\eta)}{2k \ell n}   [ \eta (z_{x,y}) - \eta (z_{x,y} + \ell_i ) ] [ f(\eta) - f(\eta^{z_{x,y}, z_{x,y} + \ell_i })  ] d\nu_{b}  \Big|.
	\end{equation}
	Our goal is to exchange the particles in the bond $\{z_{x,y} ,z_{x,y}  + \ell_i\}$. In a similar way as it is done in Lemma 6.5 of \cite{CDG}, For every $x \in \mathbb{Z}$, $y \in \{1, \ldots, \ell\}$ and every $i \in \{0, \ldots, m-1\}$, we denote
	\begin{align*}
		\Omega_i(x,y) := \Big\{\eta \in \Omega: \overrightarrow{\eta}^{\ell}(z_{x,y})\geq \frac{m}{\ell}\Big\} \cup \Big\{\eta \in \Omega: \overrightarrow{\eta}^{\ell}(z_{x,y} + \ell_{i})\geq \frac{m}{\ell}\Big\}.
	\end{align*}
	Thus, \eqref{expr3b} is bounded from above by
	\begin{equation}\label{expr3c1}
		\begin{split}
			\Big| \sum_{x} \Phi_{n}(s,\tfrac{x}{n} ) \sum_{i=0}^{k-1}    \sum_{y =1}^{\ell}  \int_{\Omega - \Omega_i(x,y)}& \frac{ \tilde{B}^{\ell}_{x,j}(\eta) }{2k \ell n}   [ \eta (z_{x,y}) - \eta (z_{x,y} + \ell_i ) ] \times\\&\times [ f(\eta) - f(\eta^{z_{x,y}, z_{x,y} + \ell_i })  ] \, d\nu_{b} \, \Big|\end{split}
	\end{equation}
	\begin{equation}\label{expr3c2}
		\begin{split}
			+\Big| \sum_{x} \Phi_{n}(s,\tfrac{x}{n} ) \sum_{i=0}^{k-1}    \sum_{y = 1}^{\ell}  \int_{\Omega_i(x,y)}& \frac{\tilde{B}^{\ell}_{x,j}(\eta)}{2k \ell n}  [ \eta (z_{x,y}) - \eta (z_{x,y} + \ell_i ) ] \times
			\\&\times[ f(\eta) - f(\eta^{z_{x,y}, z_{x,y} + \ell_i })  ] \, d\nu_{b} \, \Big|.
		\end{split}
	\end{equation}
	We note that for $\eta \in \Omega - \Omega_i(x,y)$, we have $ \eta (z_{x,y}) = \eta (z_{x,y} + \ell_i )=0$ for at least $\ell - 2m + 2$ values of $y \in \{1, \ldots, \ell\}$. Plugging this with $\frac{1}{n}\sum_{x} | \Phi_{n}(s,\tfrac{x}{n} )| \leq M_1$, $|\eta( \cdot )|\leq 1$ (which leads to $|\tilde{B}^{\ell}_{x,j}(\eta)| \leq 1$) and $f$ being a density with respect to $\nu_{b}$, we get that \eqref{expr3c1} is bounded from above by a constant times $1/\ell$. Due to our choice of $\ell$ and since  $\gamma>0$ , \eqref{expr3c1} vanishes as $n \rightarrow \infty$.
	
	It remains to deal with \eqref{expr3c2}, where we want to go from $\eta_{0,x,i,y}:=\eta$ to $\eta^{z_{x,y},z_{x,y} + \ell_i}$. If $\overrightarrow{\eta}^{\ell}(z_{x,y})\geq m/ \ell$,  the strategy is the following: for any configuration $ \eta \in \Omega_i(x,y)$, denote by $x_1, x_2, \ldots, x_m$ the positions of the $m$ particles inside the box $\{z_{x,y}+1, \ldots, z_{x,y} + \ell \}$ closest to $z_{x,y}$. With at most $m \ell$ nearest-neighbor jumps, we can move the particles at $x_1, x_2, \ldots, x_m$ to $z_{x,y}+1, z_{x,y}+2, \ldots, z_{x,y}+m$, respectively. 
	
	On the other hand, if $\eta\in\Omega_i(x,y)$ and $\overrightarrow{\eta}^{\ell}(z_{x,y}) < m/ \ell$ then necessarily we have $\overrightarrow{\eta}^{\ell}(z_{x,y}+\ell_{i})\geq  m / \ell \frac{k}{\ell}$; in this case, denote by $x_1, x_2, \ldots, x_m$ the positions of the $m$ particles inside the box $\{z_{x,y}+\ell_i+1, \ldots, z_{x,y} +\ell_i +  \ell \}$ closest to $z_{x,y}$. With at most $m \ell$ nearest-neighbor jumps, we can move the particles at $x_1, x_2, \ldots, x_m$ to $z_{x,y}+\ell_i+1, z_{x,y}+\ell_i+2, \ldots, z_{x,y}+\ell_i+m$, respectively. 
	
	In both cases, we denote the configuration with the group of at least $m$ particles in consecutive sites next to $z_{x,y}$ (resp. $z_{x,y} + \ell_i $) by $\eta_{1,x,i,y}$. Then, we exchange the particles in the bond $\{z_{x,y}, z_{x,y} + \ell_i \}$ by applying Lemma \ref{mpl}. At this point, our configuration is  $\eta_{2,x,i,y}:=(\eta_{1,x,i,y})^{z_x+y,z_x+y+\ell_i}$. Finally, we use nearest-neighbor jumps in order to bring the $m$ auxiliary particles back to their initial positions $x_1, x_2, \ldots, x_m$. We observe that our configuration now is exactly $\eta_{3,x,i,y}:=\eta^{z_{x,y},z_{x,y}+\ell_i}$.
	% Hence, we can write
	%\begin{align*}
	%f(\eta) - f(\eta^{z_{x,y}, z_{x,y}+ \ell_i}) = \sum_{r=0}^2 [f(\eta_{r,x,i,y}) - f(\eta_{r+1,x,i,y}) ].
	%\end{align*} 
	In this way, \eqref{expr3c2} is bounded from above by the sum of
	\begin{equation} \label{expr4c}
		\begin{split}
			\Big|  \sum_{x} \Phi_{n}(s,\tfrac{x}{n} ) &\sum_{i=0}^{k-1}    \sum_{y = 1}^{\ell}  \int_{\Omega_i(x,y)} \frac{\tilde{B}^{\ell}_{x,j}(\eta)}{2k \ell n}   [ \eta (z_{x,y}) - \eta (z_{x,y} + \ell_i ) ] \\
			\times & \big( [f(\eta) - f( \eta_{1,x,i,y} ) ] + [f ( \eta_{2,x,i,y} )  - f( \eta_{3,x,i,y} )  ] \big) \, d\nu_{b} \, \Big|,
		\end{split}
	\end{equation}
	\begin{equation} \label{expr6c}
		\begin{split}
			+\Big| \sum_{x} \Phi_{n}(s,\tfrac{x}{n} ) &\sum_{i=0}^{k-1}    \sum_{y = 1}^{\ell}  \int_{\Omega_i(x,y)} \frac{\tilde{B}^{\ell}_{x,j}(\eta)}{2k \ell n}   [ \eta (z_{x,y}) - \eta (z_{x,y} + \ell_i ) ] \\
			\times &  [f ( \eta_{1,x,i,y} )  - f( \eta_{2,x,i,y} )  ]  \, d\nu_{b} \, \Big|,
		\end{split}
	\end{equation}
	In order to treat \eqref{expr4c}, we write 
	\begin{align*}
		[f(\eta) - f( \eta_{1,x,i,y} ) ] + [f ( \eta_{2,x,i,y} )  - f( \eta_{3,x,i,y} )  ]  = \sum_{r \in I_{x,i,y}^{NN}} [f(\eta^{(r-1)})-f(\eta^{(r)}) ],
	\end{align*}
	where (for every fixed $x,i,y$) $I_{x,i,y}^{NN}$ is the set of bonds in which we use nearest-neighbor jumps. From Young's inequality and \eqref{boundrep}, \eqref{expr4c} is bounded from above by a constant times
	\begin{align*} 
		%&\Big| \frac{1}{4 m \ell n} \sum_{x } \Phi_n(s, \tfrac{x}{n} ) \sum_{i=0}^{m-1}    \sum_{y = 1}^{\ell} \sum_{r \in I_{x,i,y}^{NN}}\frac{1}{A_{NN}} \int_{\Omega_i(x)} [\tilde{B}_{x,j}(\eta)]^2 [\eta(z_x+y)-\eta(z_x+y+\ell_i)]^2 \big[ \sqrt{f(\eta^{(r-1)})}+\sqrt{f(\eta^{(r)})}\big]^2  d\nu_{b} \Big|\\
		%+& \Big| \frac{1}{4 m \ell n} \sum_{x } \Phi_n(s, \tfrac{x}{n} )   \sum_{i=0}^{m-1}    \sum_{y = 1}^{\ell}  \sum_{r \in I_{x,i,y}^{NN}} A_{NN} \int_{\Omega_i(x)}  \big[\sqrt{f(\eta^{(r-1)})}-\sqrt{f(\eta^{(r)})}\big]^2 d\nu_{b}  \Big| \\
		&    \frac{ \ell M_1}{A} + \frac{ A M_2  }{ k \ell n} \sum_{i=0}^{k-1}   \int_{\Omega}\sum_{x } \sum_{y = 1}^{\ell}  \sum_{r \in I_{x,i,y}^{NN}} \big[\sqrt{f(\eta^{(r-1)})}-\sqrt{f(\eta^{(r)})}\big]^2 \, d\nu_{b},
	\end{align*}
	for any $A>0$. Above we used the fact that the number of bonds of $I_{x,i,y}^{NN}$ is bounded by $C \ell$ (where $C$ is a positive constant independent of $x$, $i$ and $y$), $|\eta(\cdot)| \leq 1$ and $f$ is a density with respect to $\nu_{b}$. We observe that in the triple summation inside the integral over $\Omega$ above (here we fix $i \in \{0, \ldots, k-1\}$), the number of times that a fixed bond $\{z_0-1, z_0\}$ is counted is at most of order $\ell^2$. Thus, From \eqref{defD}, \eqref{expr4c} is bounded from above by a constant times
	\begin{align} \label{expr8b}
		\frac{ \ell M_1}{A} + \frac{A M_2}{n }\ell \,\mcb D (\sqrt{f},\nu_{b}), 
	\end{align}
	for any $A > 0$. Next, we treat \eqref{expr6c}, which deals mostly with long jumps. With a similar reasoning as we did with \eqref{expr4c}, \eqref{expr6c} is bounded from above  by
	\begin{align*}
		%& \Big| \frac{1}{4 m \ell n} \sum_{x } \Phi_n(s, \tfrac{x}{n} ) \sum_{i=0}^{m-1}    \sum_{y = 1}^{\ell} \frac{1}{A} \int_{\Omega_i(x)}[\tilde{B}_{x,j}(\eta)]^2 [\eta(z_x+y)-\eta(z_x+y+\ell_i)]^2  \big[ \sqrt{f( \eta_{1,x,i,y} )}  + \sqrt{ f ( \eta_{2,x,i,y} )  }  \big]^2  d\nu_{b} \Big| \\
		%+ & \Big| \frac{A}{4 m \ell n}\sum_{x } \Phi_n(s, \tfrac{x}{n} ) \sum_{i=0}^{m-1}    \sum_{y = 1}^{\ell}  \int_{\Omega_i(x)}  \big[ \sqrt{f( \eta_{1,x,i,y} )}  - \sqrt{ f \big( ( \eta_{1,x,i,y} )^{z_x+y,z_x+y+\ell_i}  \big) }  \big]^2  d\nu_{b} \Big| \\
		& \frac{M_1}{\tilde{A}} + \frac{\tilde{A}  M_2 }{k \ell n} \sum_{i=0}^{k-1}   \sum_{y = 1}^{\ell} \sum_{x } \int_{\Omega_i(x)}  \big[ \sqrt{f( \eta_{1,x,i,y})}  - \sqrt{ f \big( ( \eta_{1,x,i,y})^{z_x+y,z_x+y+\ell_i}  \big) }  \big]^2 \,  d \nu_{b} ,
	\end{align*} 
	for every $\tilde{A}>0$. Above we made use of \eqref{boundrep}. Applying Lemma \ref{mpl}, we conclude that \eqref{expr6c} is bounded from above  by a constant times
	\begin{align} \label{expr10b}
		\frac{M_1}{\tilde{A}} + \frac{\tilde{A} M_2 }{ k \ell n} \sum_{i=0}^{k-1}   \sum_{y = 1}^{\ell} (\ell_i)^{\gamma}  \mcb D(\sqrt{f},\nu_{b}) \leq \frac{M_1}{\tilde{A}} + \tilde{A} M_2 \varepsilon^{\gamma} m^{\gamma} n^{\gamma-1} \mcb  D (\sqrt{f},\nu_{b}), 
	\end{align}
	for any $\tilde{A} >0$. Above we used the fact that $\ell_i \leq m k \ell = m \varepsilon n$. Since $\ell=\varepsilon n^{\gamma/2}$, by choosing $A=n^{\gamma} \varepsilon^{\gamma/2} (2 M_2 \ell)^{-1}$ in \eqref{expr8b} and $\tilde{A}= \varepsilon^{-\gamma/2} (2 M_2 m^{\gamma} )^{-1}$ in \eqref{expr10b}, from \eqref{bound} we conclude that the display in \eqref{expect1b} is bounded from above by a constant times
	\begin{align*}
		M_1 M_2 \ell^2 \varepsilon^{-\gamma/2} n^{-\gamma} + M_1 M_2 \varepsilon^{\gamma/2} = M_1 M_2 ( \varepsilon^{2- \frac{\gamma}{2} } + \varepsilon^{\gamma/2} ),
	\end{align*}
	which vanishes for any $\gamma \in (0, 2)$ when $\varepsilon \rightarrow 0^{+}$, ending the proof.
\end{proof}

\quad

\thanks{ {\bf{Acknowledgements: }}
P.C. was funded by the Deutsche Forschungsgemeinschaft (DFG, German Research Foundation) under Germany's Excellence Strategy - EXC 2047/1- 390685813. 

P.G. thanks  FCT/Portugal for financial support through the
projects UIDB/04459/2020 and UIDP/04459/2020.   This project has received funding from the European Research Council (ERC) under  the European Union's Horizon 2020 research and innovative programme (grant agreement   n. 715734).

\quad

\Addresses

\end{document}